\documentclass[12pt, a4paper]{amsart}

\usepackage[T1]{fontenc}
\usepackage{textcomp}
\usepackage{relsize}
\usepackage{setspace}
\usepackage{macros}
\usetikzlibrary{positioning,shapes,fit,arrows}
\usepackage{pgf,tikz,pgfplots}
\usepackage{outlines}
\usepackage{float}
\usepackage{cancel}

\allowdisplaybreaks

\def \perm{\mathfrak{P}}
\def \CP{\mathrm{CP}}
\def \QP{\mathrm{QP}}
\def \l{\langle}
\def \r{\rangle}

\pgfplotsset{compat=1.14}

\title{On the topology of bi-cyclopermutohedra}
\author[P\, Deshpande]{Priyavrat Deshpande}
\address{Chennai Mathematical Institute, India}
\email{pdeshpande@cmi.ac.in}
\author[N\, Manikandan]{Naageswaran Manikandan}
\address{Chennai Mathematical Institute, India}
\email{naageswaran@cmi.ac.in}
\author[A\, Singh]{Anurag Singh}
\address{Chennai Mathematical Institute, India}
\email{anuragsingh@cmi.ac.in}
\thanks{PD is partially funded by the MATRICS grant. AS is partially funded by a grant from Infosys Foundation}
\date{}

\begin{document}
\begin{abstract}
Motivated by the work of Panina and her coauthors in on cyclopermutohedron we study a poset whose elements correspond to equivalence classes of partitions of the set $\{1,\cdots, n+1\}$ up to cyclic permutations and orientation reversion.
This poset is the face poset of a regular CW complex which we call bi-cyclopermutohedron and denote it by $\mathrm{QP}_{n+1}$. 
The complex $\mathrm{QP}_{n+1}$ contains subcomplexes homeomorphic to moduli space of certain planar polygons with $n+1$ sides up to isometries. 
In this article we find an optimal discrete Morse function on $\mathrm{QP}_{n+1}$ and use it to compute its homology with $\mathbb{Z}$ as well as $\mathbb{Z}_2$ coefficients. 
\end{abstract}
\maketitle

\section{Introduction}

A planar mechanical linkage is a mechanism consisting of $n$ metal bars of fixed lengths $l_0,\dots, l_n$ connected by revolving joints, that can rotate full $360^{\circ}$, forming a closed polygonal chain. 
These linkages are modelled by closed piece-wise linear paths in $\mathbb{R}^2$ called planar polygons with specified side lengths. 


\begin{definition}
Consider a length vector $\ell :=(l_0,l_1,\dots,l_n) \in \mathbb{R}_+^{n+1}$ that prescribes side lengths of planar $(n+1)$-gons. 
The moduli space of such polygons viewed up to orientation-preserving isometries of $\mathbb{R}^2$ is:
$$\mathscr{M}_\ell := \{(u_0,\dots,u_n) \in (S^1)^{n+1} \mid \sum_{i=0}^{n} l_iu_i=0\}/\mathrm{SO}(2).$$
The moduli space of $(n+1)$-gons viewed up to the action of all isometries is: 
$$\overline{\mathscr{M}}_\ell := \{(u_0,\dots,u_n) \in (S^1)^{n+1}\mid \sum_{i=0}^{n} l_iu_i=0\}/\mathrm{O}(2).$$
\end{definition}

Geometrically, the elements of $\mathscr{M}_\ell$ represent closed piecewise linear paths that differ either by a rotation or a translation (or both). 
Similarly, the elements of $\overline{\mathscr{M}}_\ell$ represent closed piecewise linear paths that differ in addition by a reflection.

A length vector $\ell$ is \emph{generic} if no configuration in $\mathscr{M}_\ell$ fits a straight line. 
Let $ [n+1]:= \{1,\dots, n+1\}$. 
A subset $I\subset [n+1]$ is \emph{short with respect to} $\ell$ (or just $\ell$-short) if 
\[\sum_{i\in I} l_i  < \sum_{j \not \in I} l_j.\] 
It is known (see \cite{Farber2007}) that when $\ell$ is generic the corresponding moduli space $\mathscr{M}_\ell$ is a smooth, closed and orientable $(n-2)$-manifold. 
In rest of this paper we deal only with generic length vectors. 

Recall that a linearly ordered partition of $[n+1]$ in to $k$ blocks is a $k$-tuple $(I_1,\dots,I_k)$ of nonempty subsets of $[n]$ such that they are mutually disjoint and their union is $[n]$. 
The set of all these ordered partitions forms a lattice under the refinement order. 
\begin{definition}
	A \textbf{cyclically ordered partition} of the set $[n+1]$ is an equivalence class of linearly ordered partitions of $[n+1]$ with the relation that two ordered partitions are equivalent if one can be obtained from the other by a cyclic permutation of its blocks.
\end{definition}

For example, $(I_1, \dots, I_k), (I_2, \dots, I_k, I_1)$ and $ (I_k, I_1, \dots, I_{k-1})$ etc. are equivalent. 
\begin{remark}
When dealing with cyclically ordered partitions, we will choose the representative in which the block containing $n+1$ appears last when written linearly. 
\end{remark}

In \cite{Panina2012}, Panina gave a regular CW structure on $\mathscr{M}_\ell$ such that the $k$-cells  correspond to cyclically ordered partitions of $[n+1]$ into $(k+3)$ blocks, where each block of the partition is $\ell$-short. 
The boundary relations are given by the partition refinement.

Motivated by above, a CW complex called \emph{cyclopermutohedron} and denoted $\CP_{n+1}$ was introduced by G. Panina in \cite{panina2015cyclopermutohedron}. 
It is an $(n-2)$-dimensional regular CW complex whose $k$-cells are labeled by cyclically ordered partitions of $[n+1]$ into $(n+1-k)$ non-empty parts, where  $(n+1-k)>2$. 
The boundary relations in the complex correspond to the orientation preserving refinement of partitions.

The cyclopermutohedron is a "universal object" for moduli spaces of polygonal linkages i.e., given a generic length vector $\ell$ the complex $\CP_{n+1}$ contains a subcomplex homeomorphic to  $\mathscr{M}_\ell$.
Interestingly, the complex $\CP_{n+1}$ is neither a topological ball (possibly) nor a wedge of spheres; its homotopy type is much more complicated. 
In \cite{nekrasov2016cyclopermutohedron}, using discrete Morse theory, the authors showed that $H_i(\CP_{n+1})$ are torsion free for $i\geq 0$ and computed their Betti numbers. 
We recall their result:

\begin{theorem}[{\cite[Theorem 2]{nekrasov2016cyclopermutohedron}}]
The homology groups of $\CP_{n+1}$ are torsion free and 
\[ \mathrm{rank}(H_i(\CP_{n+1})) =
	\begin{cases} 
	2^n + \frac{2^n-3n-2}{2}, & i=n-2; \\
	\binom{n}{i}, & 0 \leq i < n-2; \\
	0, & \hbox{otherwise}.
	\end{cases}
\]
\end{theorem}

The moduli space $\mathscr{M}_\ell$ admits a free $\mathbb{Z}_2$ action, wherein each polygon is mapped to its reflection about the $X$-axis. 
The quotient under this action is precisely $\overline{\mathscr{M}}_\ell$.
The space $\CP_{n+1}$ mimicking this action also admits a free $\mathbb{Z}_2$ action. 
The quotient space $\CP_{n+1}/\mathbb{Z}_2$ will be called \emph{bicyclopermutohedron} and denoted by $\QP_{n+1}$.
This quotient is the universal object for the moduli spaces $\overline{\mathscr{M}}_\ell$ in the same sense as described above.

The aim of this paper is to compute homology of $\QP_{n+1}$.
The statements of our main results are:
\begin{theorem}
The $\mathbb{Z}_2$-homology of $\QP_{n+1}$ is given as follows
\[ H_i(\QP_{n+1}, \mathbb{Z}_2) =
	\begin{cases}
	\displaystyle \bigoplus_{\xi(n,i)}\mathbb{Z}_2, & 0 \leq i \leq n-2;\\
	0, & \mathrm{otherwise}.
	\end{cases}\]
\end{theorem}

\begin{theorem}\label{thm2intro}
The $\mathbb{Z}$-homology of $\QP_{n+1}$ is given as follows.

If n is even, then 
\[ H_i(\QP_{n+1},\mathbb{Z}) =
	\begin{cases}
    \displaystyle \bigoplus_{\xi(n,i)}\mathbb{Z}_2 , & \mathrm{if\ } i \mathrm{\ is\ odd\ and\ } 0 \leq i \leq n-2 ;\\
	\displaystyle \bigoplus_{\binom{n}{i}}\mathbb{Z} , & \mathrm{if\ } i \mathrm{\ is\ even\ and\ } 0 \leq i \leq n-2;\\
	0, & \mathrm{otherwise.}
	\end{cases}
\]

If n is odd, then

\[ H_i(\QP_{n+1},\mathbb{Z}) =
	\begin{cases}
	\displaystyle \bigoplus_{\xi(n,i)}\mathbb{Z}, & \mathrm{if\ } i = n-2;\\
	\displaystyle \bigoplus_{\xi(n,i)}\mathbb{Z}_2, & \mathrm{if\ } i \mathrm{\ is\ odd\ and\ } 0 \leq i < n-2 ;\\
	\displaystyle \bigoplus_{\binom{n}{i}}\mathbb{Z} , & \mathrm{if\ } i \mathrm{\ is\ even\ and\ } 0 \leq i \leq n-2;\\
	0, & \mathrm{otherwise.}
	\end{cases}
\]
\end{theorem}
In both the above theorems $\xi(n, i)$ denotes the partial sum $\sum_{k=0}^i \binom{n}{k}$ of binomial coefficients. 

The article is organized as follows. 
\Cref{sec2} is lists the relevant results from discrete Morse theory. 
We mainly recall the construction of the Morse complex and a generic formula for the boundary homomorphisms. 
Next, in \Cref{sec3} we introduce the cyclopermutohedron; it is a regular CW complex.
The new result here is a closed form formula for the degree of the attaching homeomorphisms in terms of underlying combinatorics. 
In \Cref{sec4} we recall the discrete Morse function on this complex first described in \cite{nekrasov2016cyclopermutohedron}.
We also give a different proof of the fact that this Morse function is optimal, i.e., in the Morse complex all boundary homomorphisms vanish.
The reproof is necessary because of the strategy we employ can be reused in our new results. 
Next in \Cref{sec5} we introduce the new object the bi-cyclopermutohedron (denoted $\QP_{n+1}$) and define a discrete Morse function on it. 
Using involved combinatorial arguments we show that the number of paths between any two critical cells is either zero or two. 
As a result we compute the homology with $\mathbb{Z}_2$ coeffecients, see \Cref{mod2thm}.
In  \Cref{sec6} we turn our attention to the integral homology of $\QP_{n+1}$. 
The boundary maps in the Morse complex do not vanish. 
As a result the computations are quite intricate. 
However using the combinatorial formula for the degree of the attaching maps, a clear description of the gradient paths and the good path lemma (\Cref{good path lemma}) we explictly compute the boundary homomorphisms.
In \Cref{mainthm} we compute the integer homology of $\QP_{n+1}$ and also show that there is only $2$-torsion. 

\section{Discrete Morse Theory}\label{sec2}

In this section we quickly recall some results from discrete Morse theory that we need, see \cite{Forman1998} for more details.
Recall that the homeomorphism type of a regular CW complex is completely determined by its face poset. 
Hence, for two cells $\sigma, \tau$ by $\sigma<\tau$ we mean $\sigma\subset\overline{\tau}$.
Moreover, the superscript notation $\sigma^p$ denotes the dimension of the cell. 

\begin{definition}
	A \emph{discrete vector field} $V$ on a regular CW complex $\cK$ is a collection of pairs $(\sigma^p, \tau^{p+1})$ where $\sigma < \tau$ such that each cell is in at most one pair of $V$.
\end{definition}

\begin{definition}
Given a discrete vector field $V$ on $\cK$, a $V$-path is a sequence of cells
$$\sigma^p_1, \tau^{p+1}_1,\dots, \sigma^p_t, \tau^{p+1}_t, \sigma^p_{t+1}$$
such that, for each $i \in [t]$, $\sigma_{i+1} \neq \sigma_i$, $(\sigma_i, \tau_i)$ is a pair in $V$ and $\tau_i > \sigma_{i+1}$.  
A path is called \emph{closed} if $\sigma_1=\sigma_{t+1}$.
\end{definition}

\begin{definition}
    A discrete vector field is called a \emph{discrete Morse function} if there are no closed $V$-paths. In this case, $V$-paths are called gradient paths.
\end{definition}

For $\sigma^{p} < \tau^{p+1}$ in $\cK$, the \emph{incidence number} $[\tau:\sigma]$ is the degree of the attaching homeomorphism.
Consider two distinct $p$-cells $\sigma_1, \sigma_2$ and a $(p+1)$-cell $\tau$ such that $\sigma_1< \tau$ and $\sigma_2<\tau$. 
Fixed orientations on $\sigma_1$ and $\tau$, induce an orientation on $\sigma_2$ so that $[\tau:\sigma_1]\cdot [\tau: \sigma_2]=-1$.
Let $V$ be a discrete Morse function on $\cK$ and let $C:=\sigma_1, \tau_1, \dots , \sigma_t, \tau_t, \sigma_{t+1}$ be a gradient path. 
An orientation on $\sigma_1$ induces an orientation on each $\sigma_i$ in turn, and, in particular, on $\sigma_{t+1}$. 
Define $w(C)=1$ if the fixed orientation on $\sigma_1$ induces the fixed orientation on $a_{t+1}$, and $w(C)=-1$ otherwise.

A cell $\sigma$ is \emph{critical} for a discrete Morse function $V$, if it is not paired in $V$. 
Let $\mathcal{M}_p$ denote the free abelian group generated by the critical $p$-cells. 
The Morse complex $\mathcal{M}_{\bullet}$ on $\cK$ is defined as follows,
$$ \mathcal{M}_{\bullet}: 0\longrightarrow \dots  \mathop{\longrightarrow}^{\mathrm{\tilde{\partial}}} \mathcal{M}_{p+1} \mathop{\longrightarrow}^{\mathrm{\tilde{\partial}}} \mathcal{M}_{p} \mathop{\longrightarrow}^{\mathrm{\tilde{\partial}}}  \dots  \mathop{\longrightarrow}^{\mathrm{\tilde{\partial}}} \mathcal{M}_0 \rightarrow 0;$$
the boundary homomorphism is given by 
\begin{align}
    \tilde{\partial}\tau^{p+1} &= \sum_{\sigma^p < \tau} \langle \tilde{\partial}\tau, \sigma \rangle\; \sigma,\label{bdryeq1}\\
    \langle \tilde{\partial}\tau, \sigma \rangle &:= \sum_{\tilde{\sigma}^p < \tau} [\tau:\tilde{\sigma}] \sum_{c \in \Gamma(\tilde{\sigma}, \sigma)} w(c)\label{bdryeq2},
\end{align}
where, $\Gamma(\tilde{\sigma}, \tau)$ denotes the set of all gradient paths from $\tilde{\sigma}$ to $\tau$.

\section{the cyclopermutohedron}\label{sec3}
An important problem in topological combinatorics is to compute topological invariants of combinatorially defined complexes.
For example, given a poset one would like to know whether its geometric realization is a PL-manifold, what is its homotopy type or what are the homology groups?
Interestingly, for a lot of posets their geometric realizations are homotopy equivalent to a wedge of spheres. 
Cohen-Macaulay property in general and Shellability in particular are extensively studied topics in this area. 
See for example \cite{wachs2006poset}.

For instance, consider the poset of unordered partitions of $[n]$. 
Its geometric realization has the homotopy type of wedge of sphere. 
On the other hand geometric realization of the poset of ordered partitions of $[n]$ is the permutohedron $\perm_n$.
Recall that $\perm_n$, is an $(n-1)$-dimensional convex polytope in  $\mathbb{R}^n$ formed by taking the convex hull of all points that are obtained by permuting the coordinates of the point $(1, 2,\dots, n)$.

A natural question now is what would happen if one were to replace the linear ordering by the cyclic ordering? 
What is the homotopy type of the resulting poset?
To the best of our knowledge the poset of cyclically ordered partitions first appeared in the work of Panina. 
She constructed a CW complex, called the \emph{cyclopermutohedron} and denoted $\CP_{n+1}$, whose face poset is isomorphic to the poset of cyclically ordered partitions. 
She showed that it can not be realized as a polytope in any Euclidean space. 
However, it can be realized as a \emph{virtual polytope} (see \cite{Panina2012, panina2015cyclopermutohedron}). 
Intuitively, a virtual polytope is a formal difference of two polytopes. 
However, the reader will see that each closed cell is combinatorially equivalent to the product of at most $3$  permutohedra of appropriate dimension.  

With her coauthors Panina also computed the homology groups of $\CP_{n+1}$ in \cite{nekrasov2016cyclopermutohedron}.
This was done by finding an optimal discrete Morse function and then computing the number of critical cells. 
In the remaining of this section we recall some relevant results proved in \cite{nekrasov2016cyclopermutohedron} and provide different proof of vanishing of boundary maps in the Morse complex.  
In particular, we find a closed form formula for the degree of the attaching homeomorphisms. 
We use it to establish ``the good path lemma", both these results are crucial for theorems proved in subsequent sections.


\begin{definition}
The regular CW complex cyclopermutohedron $\CP_{n+1}$ is defined as:
\begin{itemize}
	\item For $k = 0, 1 ,\dots, n-2$, the k-cells of $\CP_{n+1}$ are labeled by (all possible) cyclically ordered partitions of the set $[n + 1]$ into $(n-k+1)$ non-empty parts.
	
	\item A (closed) cell $F$ contains a cell $F^{'}$ whenever the label of $F^{'}$ refines the label of $F$. Here, by refinement we mean orientation preserving refinement.
\end{itemize}
\end{definition}
\cref{fig: Complex n=4} depicts the complex for $n=3$.
\begin{figure}[!ht]
\centering
\definecolor{zzttqq}{rgb}{0.6,0.2,0.}
\definecolor{ttqqqq}{rgb}{0.2,0.,0.}
\begin{tikzpicture}[line cap=round,line join=round,>=triangle 45,x=1.0cm,y=1.0cm]
\clip(1.,0.) rectangle (11.,11.);
\draw[line width=1.pt,color=zzttqq] (6.,3.) -- (8.,4.) -- (8.,6.) -- (6.,7.) -- (4.,6.) -- (4.,4.) -- cycle;
\draw[line width=1 pt,color=zzttqq] (6.,7.) -- (8.,4.) -- (4.,4.) -- cycle;
\draw[line width=1 pt,color=zzttqq] (4.,6.) -- (6.,3.) -- (8.,6.) -- cycle;
\draw [line width=1.2pt,color=zzttqq] (6.,3.)-- (8.,4.);
\draw [line width=1.2pt,color=zzttqq] (8.,4.)-- (8.,6.);
\draw [line width=1.2pt,color=zzttqq] (8.,6.)-- (6.,7.);
\draw [line width=1.2pt,color=zzttqq] (6.,7.)-- (4.,6.);
\draw [line width=1.2pt,color=zzttqq] (4.,6.)-- (4.,4.);
\draw [line width=1.2pt,color=zzttqq] (4.,4.)-- (6.,3.);
\draw [line width=1.2pt,color=zzttqq] (6.,7.)-- (8.,4.);
\draw [line width=1.2pt,color=zzttqq] (8.,4.)-- (4.,4.);
\draw [line width=1.2pt,color=zzttqq] (4.,4.)-- (6.,7.);
\draw [line width=1.2pt,color=zzttqq] (4.,6.)-- (6.,3.);
\draw [line width=1.2pt,color=zzttqq] (6.,3.)-- (8.,6.);
\draw [line width=1.2pt,color=zzttqq] (8.,6.)-- (4.,6.);
\draw [line width=1pt] (5.992847529397043,8.339721192317038) circle (0.5cm);
\draw [line width=1pt] (8.977569784356508,6.816903715296904) circle (0.5cm);
\draw [line width=1pt] (8.997874017383442,3.14183753742165) circle (0.5cm);
\draw [line width=1pt] (6.013151762423979,1.6799327594823228) circle (0.5cm);
\draw [line width=1pt] (2.9878210414106445,6.816903715296904) circle (0.5cm);
\draw [line width=1pt] (2.9878210414106445,3.1824460034755204) circle (0.5cm);
\draw (5.6,1.1339300824662997) node[anchor=north west] {\{1\}};
\draw (4.557343036332395,1.9943448439324962) node[anchor=north west] {\{2\}};
\draw (5.5308698365027745,2.793431837623781) node[anchor=north west] {\{3\}};
\draw (6.437508675377335,1.9674568826366776) node[anchor=north west] {\{4\}};
\draw (7.7008794701464089,3.446294753906703) node[anchor=north west] {\{1\}};
\draw (8.534769656451194,7.882808367716779) node[anchor=north west] {\{1\}};
\draw (5.503981875206955,9.461646238986804) node[anchor=north west] {\{1\}};
\draw (1.70127793324965193,7.129945451433857) node[anchor=north west] {\{1\}};
\draw (2.5731940939627174,2.693431837623781) node[anchor=north west] {\{1\}};
\draw (8.532150163447,2.612767953736325) node[anchor=north west] {\{2\}};
\draw (8.5209462930288,6.26953068996766) node[anchor=north west] {\{2\}};
\draw (4.724230997628213,8.6356712839997) node[anchor=north west] {\{2\}};
\draw (2.51302100752614,7.85592040642096) node[anchor=north west] {\{2\}};
\draw (2.546306132666899,4.252933592781262) node[anchor=north west] {\{2\}};
\draw (8.515433540338649,4.172269708893806) node[anchor=north west] {\{3\}};
\draw (7.68190674016827,7.07616952884222) node[anchor=north west] {\{3\}};
\draw (5.503981875206955,7.875256522533504) node[anchor=north west] {\{3\}};
\draw (2.4925302100752614,6.350194573855116) node[anchor=north west] {\{3\}};
\draw (1.7127793324965193,3.50007067649834) node[anchor=north west] {\{3\}};
\draw (3.3991690489498223,3.50007067649834) node[anchor=north west] {\{4\}};
\draw (9.34896034050903,3.446294753906703) node[anchor=north west] {\{4\}};
\draw (9.34896034050903,7.103057490138038) node[anchor=north west] {\{4\}};
\draw (6.3106207140815165,8.608783322703882) node[anchor=north west] {\{4\}};
\draw (3.3091690489498223,7.07616952884222) node[anchor=north west] {\{4\}};
\draw [line width=1.2pt] (8.00296659906362,8.766110085882675) circle (0.5cm);
\draw [line width=1.2pt] (8.00296659906362,1.1926311668358802) circle (0.5cm);
\draw (7.439915088505902,2.236336495594864) node[anchor=north west] {\{1,2\}};
\draw (8.354105772563737,1.5372495019035792) node[anchor=north west] {\{3\}};
\draw (7.520578972393358,0.7037227017332014) node[anchor=north west] {\{4\}};
\draw (7.493691011097539,9.84562954231154) node[anchor=north west] {\{1\}};
\draw (7.46680304980172,8.178575941970784) node[anchor=north west] {\{3\}};
\draw (8.354105772563737,9.03899070343698) node[anchor=north west] {\{2,4\}};
\draw [->,line width=1.2pt] (7.617186171551853,8.07576616296688) -- (6.,6.);
\draw [->,line width=1.2pt] (7.576577705497983,1.8423666236978036) -- (6.806330507088282,3.403165253544141);
\begin{scriptsize}
\draw [fill=ttqqqq] (6.,3.) circle (1.0pt);
\draw [fill=ttqqqq] (8.,4.) circle (1.0pt);
\draw [fill=ttqqqq] (8.,6.) circle (1.0pt);
\draw [fill=ttqqqq] (6.,7.) circle (1.0pt);
\draw [fill=ttqqqq] (4.,6.) circle (1.0pt);
\draw [fill=ttqqqq] (4.,4.) circle (1.0pt);
\end{scriptsize}
\end{tikzpicture}
\caption{The Complex $\CP_4$}
\label{fig: Complex n=4}
\end{figure}
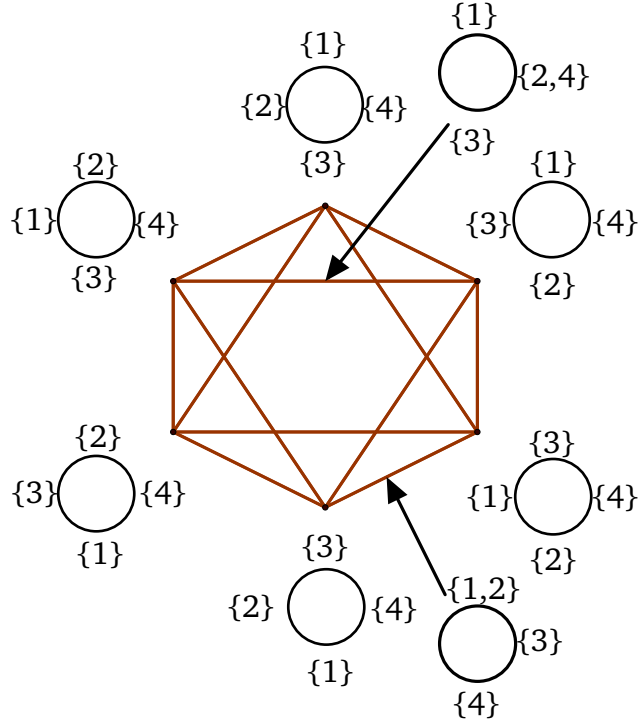

Throughout this paper we adopt following conventions and notations when dealing with ordered partitions.
\begin{enumerate}
    \item A cyclically ordered partition is represented by linear partition in which the block containing $n+1$ appears last.  
    \item A subset of $[n+1]$ containing the element $n+1$ will be called an $n+1$-set. 
    Given a partition of $[n+1]$, the letter $N$ denotes the $n+1$-set.
    \item The triangle $\nabla$ denotes (a possibly empty) string of singletons arranged in decreasing order.
    \item Given two subsets $I$ and $J$, the expression ``$I<J$'' indicates that $i < j$ for each $i \in I$ and $j \in J$. Similarly, the expression ``$k < \nabla$'' indicates that $k$ is less than the element in each singleton of $\nabla$.
    \item The set ``$I-\{m\}$'' is denoted ``$I-m$'' and the braces for the singleton will be omitted \textit{i.e.,} the block ``$\{s\}$'' is denoted by ``$s$''.
    \item The set ``$\{a_1,a_2,\dots,a_k\}$'' will be denoted ``$a_1a_2\dots a_k$'' when there is no ambiguity.
\end{enumerate}


\subsection{A formula for the incidence numbers} \label{canonical orientation}
The aim of this section is obtain a combinatorial formula describing the degree of an attaching homeomorphism in a cyclopermutohedron.
Note that the vertices of $\CP_{n+1}$ are in bijection with the elements of the group $S_n$.
Two vertices are joined by an edge whenever their labels differ by a transposition.
Given a vertex $v$ in $\alpha=(I_1,I_2,\dots,I_l)$, there are $\dim(\alpha)$ many vertices of $\alpha$ that are connected to $v$ by an edge. 
We call such vertices neighbors of $v$ and order them as follows. 
The first neighbor of $v$ is obtained by interchanging the first two entries of $v$ that belong to the same block, the second neighbor is obtained by interchanging the second two entries of $v$ that belong to the same block and so on. 
This ordering is called an \emph{orientation} of the cell with respect to vertex $v$. 

\begin{definition}
The \emph{principal vertex} $\mathrm{PV}(\alpha)$ of a cell $\alpha$ is the vertex with the label $(\hat{I_1}, \hat{I_2},\dots,\hat{I_l})$, where $\hat{I_j}$ is a partition of the set $I_j$ into singletons coming in increasing order.
The orientation of the cell $\alpha$  related to its principal vertex $\mathrm{PV}(\alpha)$ is called the \emph{canonical orientation} of $\alpha$.
\end{definition}

\begin{example}
For the cell $\alpha=(\{1\} \{2,4,5\} \{3\} \{6,7,8\})$, the principal vertex $\mathrm{PV}(\alpha)$ is given by $(\{1\} \{2\}\{4\}\{5\} \{3\} \{6\} \{7\}\{8\})$ and the $\alpha$-neighbors of $\mathrm{PV}(\alpha)$ are ordered as follows:\\
	$v_1 = (1,4,2,5,3,6,7,8)$,	$v_2 = (1,2,5,4,3,6,7,8)$,	$v_3 = (1,2,4,5,3,7,6,8)$ etc.
\end{example}

\begin{proposition}
The free and transitive action of the symmetric group $S_{n+1}$ on the vertices of $\CP_{n+1}$ preserves the canonical ordering.
\end{proposition}

\begin{proof}
Clear.
\end{proof}
\begin{example}
Here is an example that illustrates the above Proposition. Let $\sigma= (12345,6,7) \in \CP_7$, and let $w=(24) \in S_{7}$.
\begin{align*}
    v_0=(1,2,3,4,5,6,7)& \mapsto (1,4,3,2,5,6,7)\\
    v_1=(2,1,3,4,5,6,7)& \mapsto (4,1,3,2,5,6,7)\\
    v_2=(1,3,2,4,5,6,7)& \mapsto (1,3,4,2,5,6,7)\\
    v_3=(1,2,4,3,5,6,7)& \mapsto (1,2,3,4,5,6,7)\\
    v_4=(1,2,3,5,4,6,7)& \mapsto (1,4,3,5,2,6,7)
\end{align*}
\end{example}

Given a pair of cells $\tau^{p-1} <\sigma^{p}$ in $\CP_{n+1}$, let $v_{\sigma}$ and $v_{\tau}$ denote the principal vertices of $\sigma$ and $\tau$ respectively. Let $(v_1,\dots, v_p)$ be the ordering on neighbors of $v_{\sigma}$ in $\sigma$. 
Since $v_{\sigma}$ and $v_{\tau}$ also represent elements of $S_n$, there exists a permutation $g_{\sigma,\tau} \in S_n$ such that $g_{\sigma,\tau} v_{\sigma} = v_{\tau}$. Moreover, there is exactly one index $i_{\tau}\in \{1,2,3,\dots,p\}$ such that $g_{\sigma,\tau} v_{i_{\tau}}$ is not adjacent to $v_{\tau}$ in $\tau$.

Let $\Delta_m$ be the free abelian group with basis the $m$-cells $\sigma_{\alpha}^m$ of $\CP_{n+1}$. 
We define the boundary homomorphism $\partial_m: \Delta_m \rightarrow \Delta_{m-1}$ by specifying its values on the basis elements. 
Denote by $\langle \partial_m \sigma, \tau \rangle$ the coefficient of $\tau^{m-1}$ in $\partial_m \sigma$ and define it as:
\begin{align*}
    \langle \partial_m \sigma, \tau \rangle :=&\  \mathrm{\mathrm{sign}}(g_{\sigma,\tau})\cdot (-1)^{i_{\tau}}.
\end{align*}
The remainder of the section is devoted to prove that the above integer is indeed the degree of the homeomorphism that attaches the cell $\tau$ to the boundary of $\sigma$.
\begin{lemma}
	Let $\sigma =(I_1,I_2,\dots,I_k)$ and $\tau=(I_1,\dots,J_1,J_2,\dots,I_k)$ with some $p \in \{1,2,\dots,k\}$ such that $I_p=J_1 \cup J_2$. For every $i \in \{1,2,\dots,k\}$, denote  $|I_i|=r_i$. Then  $$\langle\partial \sigma, \tau\rangle= (-1)^{\sum_{i=1}^{p-1}r_{i}+|J_1|-(p-1)}\cdot \mathrm{sign}(g_{\sigma,\tau}).$$
\end{lemma}

\begin{proof}
	Without loss of generality assume $\mathrm{PV}(\sigma)= \mathrm{PV}(\tau)= (1,2,\dots,n+1)=v_0$. The neighbors of $\mathrm{PV}(\sigma)$ are ordered as follows:
	\begin{equation*}
	\begin{gathered}
	v_0=(1,2,\dots,n+1)\\
	v_1=(2,1,3,\dots,n+1)\\
	\vdots\\
	v_{r_1-1}=(1,2,\dots,r_1,r_1-1,\dots,n+1)\\
	v_{r_1}=(1,2,\dots,r_1,r_1+2,r_1+1,\dots,n+1)\\
	\vdots\\
	v_{r_{p-1}+|J_1|-(p-1)-1}=(1,2,\dots,r_{p-1}+|J_1|,r_{p-1}+|J_1|-1,\dots,n+1)\\
	v_{r_{p-1}+|J_1|-(p-1)}=(1,2,\dots,r_{p-1}+|J_1|+1,r_{p-1}+|J_1|,\dots,n+1)\\
	\vdots
	\end{gathered}
	\end{equation*}

From the list, it is clear that the index $i_\tau$ such that $v_{i_{\tau}}$ is not a vertex of $\tau$ is $r_{p-1}+|J_1|-(p-1)$, since the interchanging is consistent with only $\sigma$ and not with $\tau$. A similar argument works for the case where $\mathrm{PV}(\tau) \neq (1,2,\dots,n+1)$. Observe the fact that the missing index corresponding to the cell which 
\begin{itemize}
	\item has same partition structure (i.e., similar block structure) as $\tau$, 
	\item is contained in the boundary of $\sigma$ and 
	\item has same principal vertex as $\sigma$,
\end{itemize}

is the unique index $i_{\tau}$ such that $g_{\sigma,\tau}v_{i_{\tau}}$ is not adjacent to $v_{\tau}$ in $\tau$. Briefly, a missing index is taken to another missing index by the permutation $g_{\tau}$.
\end{proof}	

\begin{theorem}
	The composition $ \Delta_m \xrightarrow{\partial_m} \Delta_{m-1} \xrightarrow{\partial_{m-1}} \Delta_{m-2}$ is zero. 
\end{theorem}

\begin{proof} 
	Let $\sigma^{k+2}=(I_1,I_2,\dots,I_{n-1-k})$ , $\tau^{k}=(J_1,\dots,J_{n+1-k})$, then it is enough to show that $\langle \partial^2\sigma,\tau\rangle=0$. If $\tau^{k}<\sigma^{k+2}$ then $\sigma$ and $\tau$ satisfy exactly one of the following relations
	\begin{enumerate}
		\item $\exists i,j,k \in [n+1-k]$ and $\exists p \in [n-1-k]$ such that $J_i\cup J_j \cup J_k =I_p$,
		\item $\exists i,j,s,t \in [n+1-k]$ and $\exists p,q \in [n-1-k]$ such that $J_i \cup J_j =I_s$ and $J_k\cup J_l=I_t$.
	\end{enumerate}

\textbf{Case 1:} Without loss of generality assume $J_1\cup J_2 \cup J_3 = I_1$, so only $I_1$ is involved in the computation of $\langle \partial^2\sigma,\tau \rangle$. We can also assume  that $\sigma$ has the minimum number of blocks, i.e., $s=3$. Thus, $\sigma=(I_1,I_2,I_3)$ and $\tau=(J_1,J_2,J_3,J_4,J_5)$ and we have
	
	\begin{figure}[!ht]
		\centering
		\begin{tikzpicture}
		\sffamily
		\node (a0) at (0,0) {$(I_1,I_2,I_3)$};
		\node [below left=1cm of a0] (a1)  {$\omega_1=(J_1\cup J_2,J_3,J_4,J_5)$};
		\node [below right=1cm of a0] (a2) {$\omega_2=(J_1,J_2\cup J_3,J_4,J_5)$};
		\draw [blue, thick] (a0) -- (a1);
		\draw [blue, thick] (a0) -- (a2);
		\node[below=2cm of a0] (b) {$(J_1,J_2,J_3,J_4,J_5)$};
		\draw [blue, thick] (a1) -- (b);
		\draw [blue, thick] (a2) -- (b);
		
		\end{tikzpicture}
		\caption{The interval $[\tau,\sigma]$ in case 1.}
	\end{figure}
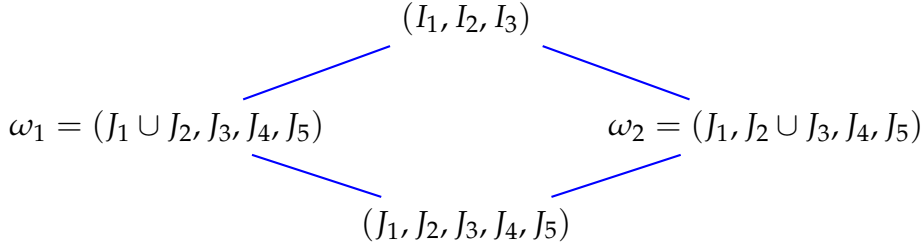
	\begin{itemize}
		\item $\langle \partial\sigma,\omega_1 \rangle= \mathrm{sign}(g_1)\cdot(-1)^{|J_1|+|J_2|}$
		\item $\langle \partial\sigma,\omega_2 \rangle=\mathrm{sign}(g_2)\cdot(-1)^{|J_1|}$
		\item $\langle \partial\omega_1,\tau \rangle=\mathrm{sign}(g_3)\cdot(-1)^{|J_1|}$
		\item $\langle \partial\omega_2,\tau \rangle=\mathrm{sign}(g_4)\cdot(-1)^{|J_1|+|J_2|-1}$
	\end{itemize} 
	where, $g_i$'s represent the permutation involved in the comparison of principal vertices.
	Note that there is a unique permutation which takes $\mathrm{PV}(\sigma)$ to $\mathrm{\mathrm{PV}}(\tau)$, so $g_3\circ g_1 = g_4\circ g_2$. This shows that $\langle \partial^2 \sigma,\tau \rangle=0$.
	
\textbf{Case 2:} Without loss of generality assume that $J_1\cup J_2= I_1$ and $J_3\cup J_4= I_2$, so only $I_1$ and $I_2$ are involved in the computation of $\langle \partial^2\sigma,\tau \rangle$. Further assume that $\sigma$ has the minimum number of blocks, i.e., $s=3$.	Thus, $\sigma=(I_1,I_2,I_3)$ and $\tau=(J_1,J_2,J_3,J_4,J_5,J_6)$ and we have the following 
	
	\begin{figure}[!ht]
		\centering
		\begin{tikzpicture}
		\sffamily
		\node (a0) at (0,0) {$(I_1,I_2,I_3)$};
		\node [below left=1cm of a0] (a1)  {$\omega_1=(J_1\cup J_2,J_3,J_4,J_5,J_6)$};
		\node [below right=1cm of a0] (a2) {$\omega_2=(J_1,J_2,J_3 \cup J_4,J_5,J_6)$};
		\draw [blue, thick] (a0) -- (a1);
		\draw [blue, thick] (a0) -- (a2);
		\node[below=2cm of a0] (b) {$(J_1,J_2,J_3,J_4,J_5,J_6)$};
		\draw [blue, thick] (a1) -- (b);
		\draw [blue, thick] (a2) -- (b);
		
		\end{tikzpicture}
		\caption{The interval $[\tau,\sigma]$ in case 2.}
	\end{figure}
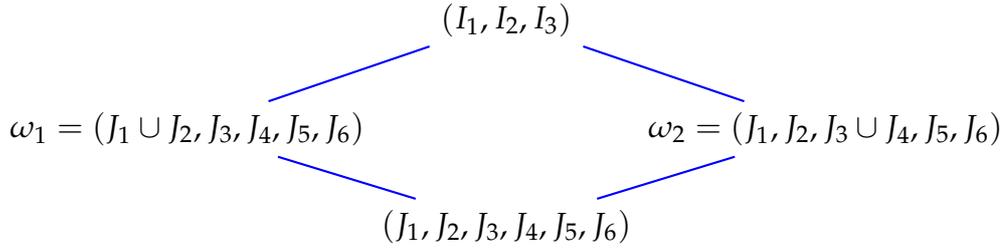
	
	\begin{itemize}
		\item $\langle \partial\sigma,\omega_1\rangle=\mathrm{sign}(g_1)\cdot(-1)^{|J_1|+|J_2|+|J_3|-1}$
		\item $\langle \partial\sigma,\omega_2\rangle =\mathrm{sign}(g_2)\cdot(-1)^{|J_1|}$
		\item $\langle \partial\omega_1,\tau \rangle=\mathrm{sign}(g_3)\cdot(-1)^{|J_1|}$
		\item $\langle \partial\omega_2,\tau \rangle=\mathrm{sign}(g_4)\cdot(-1)^{|J_1|+|J_2|+|J_3|-2}$
	\end{itemize} 
	where, $g_i$'s represent the permutation involved in the comparison of principal vertices.\\
	Note that there is a unique permutation which takes $\mathrm{PV}(\sigma)$ to $\mathrm{PV}(\tau)$, so $g_3\circ g_1 = g_4\circ g_2$. This shows that $\langle \partial^2 \sigma,\tau \rangle=0$.
\end{proof}

We can now prove the main result of this section. 

\begin{theorem}
	Let $\tau^{p-1}<\sigma^p$, then
	\begin{equation} \label{INCR}
	[\sigma:\tau]= \mathrm{sign}(g_{\sigma,\tau})\cdot (-1)^{i_{\tau}}.
	\end{equation}
	i.e., the coefficient of $\tau$ in the image of $\sigma$ under the boundary homomorphism $\partial$ is precisely the incidence number $[\sigma:\tau]$.
\end{theorem}

\begin{proof}
	We will prove this inequality using induction on the dimension of cells. If dimension of $\sigma$ is 2, then the boundary complex is exactly one of the following.
	
\begin{figure}[H]
\centering
\begin{tikzpicture}[line cap=round,line join=round,>=triangle 45,x=1.0cm,y=1.0cm]
\clip(0.,1.) rectangle (7.,6.);
\draw [line width=1.2pt] (0.72,4.08)-- (2.,5.);
\draw [line width=1.2pt] (2.,5.)-- (3.26,4.08);
\draw [line width=1.2pt] (3.26,4.08)-- (3.28,3.06);
\draw [line width=1.2pt] (3.28,3.06)-- (2.,2.18);
\draw [line width=1.2pt] (2.,2.18)-- (0.72,3.02);
\draw [line width=1.2pt] (0.72,3.02)-- (0.72,4.08);
\draw [line width=1.2pt] (5.46,4.4)-- (4.56,3.46);
\draw [line width=1.2pt] (4.56,3.46)-- (5.46,2.54);
\draw [line width=1.2pt] (5.46,2.54)-- (6.38,3.44);
\draw [line width=1.2pt] (6.38,3.44)-- (5.46,4.4);
\begin{scriptsize}
\draw [fill=blue] (0.72,4.08) circle (2.0pt);
\draw [fill=blue] (2.,5.) circle (2.0pt);
\draw [fill=blue] (3.26,4.08) circle (2.0pt);
\draw [fill=blue] (3.28,3.06) circle (2.0pt);
\draw [fill=blue] (2.,2.18) circle (2.0pt);
\draw [fill=blue] (0.72,3.02) circle (2.0pt);
\draw [fill=blue] (5.46,4.4) circle (2.0pt);
\draw [fill=blue] (4.56,3.46) circle (2.0pt);
\draw [fill=blue] (5.46,2.54) circle (2.0pt);
\draw [fill=blue] (6.38,3.44) circle (2.0pt);
\end{scriptsize}
\end{tikzpicture}
\caption{Possible boundaries of a 2-cell}
\end{figure}
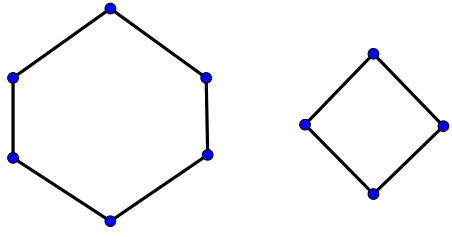
	
By computing the incidence numbers explicitly, \cref{INCR} can be proven easily.

Now assume the induction hypothesis that the \cref{INCR} is true for all cells of dimension less than or equal to $k$. Let $\sigma=(I_1,I_2 \dots I_{n-k}), \tau=(I_1,I_2 \dots J_1,J_2 \dots I_{n-k})$ with $J_1\cup J_2=I_p$ and $|I_i|=r_i$. Without loss of generality we can assume that the $\mathrm{PV}(\sigma)=(1,2,3, \dots, n+1)$.

\textbf{Step 1:} If $\tau$ has the same principal vertex as $\sigma$, in that case $J_1=(r_{p-1}+1,r_{p-1}+2,\dots, r_{p-1}+t)$ and $J_2=(r_{p-1}+t+1,r_{p-1}+t+2, \dots, r_p)$.\\ 
Let $\tilde{\tau}=(I_1,I_2, \dots, \tilde{J}_1,\tilde{J}_2, \dots, I_{n-k})$ where $\tilde{J}_1=(r_{p-1}+1,r_{p-1}+2, \dots, r_{p-1}+t-1)$ and $\tilde{J}_2=(r_{p-1}+t,r_{p-1}+t+1, \dots, r_p)$.

\begin{figure}[!ht]
	\centering
	\begin{tikzpicture}
	\sffamily
	\node (a0) at (0,0) {$\sigma=(I_1,I_2 \dots I_{n-k})$};
	\node [below left=1cm of a0] (a1)  {$\tau=(I_1,I_2 \dots J_1,J_2 \dots I_{n-k})$};
	\node [below right=1cm of a0] (a2) {$\tilde{\tau}=(I_1,I_2 \dots \tilde{J}_1,\tilde{J}_2 \dots I_{n-k})$};
	\draw [blue, thick] (a0) -- (a1);
	\draw [blue, thick] (a0) -- (a2);
	\node[below=2cm of a0] (b) {$\omega=(I_1,I_2 \dots \tilde{J}_1,\{r_{p-1}+t\},J_2 \dots I_{n-k})$};
	\draw [blue, thick] (a1) -- (b);
	\draw [blue, thick] (a2) -- (b);
	
	\end{tikzpicture}
	\caption{}
\end{figure}

By induction hypothesis we know $[\tau:\omega] \cdot [\tilde{\tau}:\omega]= 1 $. Since the square of the boundary map vanishes in the cellular chain complex, we have $[\sigma:\tau] \cdot [\sigma:\tilde{\tau}]=-1$. Let $\gamma=(1,I_1-\{1\},I_2,\dots,I_{n-k})$. Then fixing the value of $[\sigma:\gamma]$ fixes the value for every $\tau$ whose principal vertex is same as $\sigma$. we fix this value to be -1.

\textbf{Step 2:} If $\mathrm{PV}(\sigma)\neq \mathrm{PV}(\tau)$ then it is enough to consider the cells where permutation required to take one to the other is just a transposition.\\
Then $\tau=(I_1,I_2, \dots, J_1,J_2, \dots, I_{n-k})$ where  $J_1=(r_{p-1}+1,r_{p-1}+2, \dots, r_{p-1}+t+1)$, $J_2=(r_{p-1}+t,r_{p-1}+t+2, \dots, r_p)$.
Let $\tilde{\tau}=(I_1,I_2, \dots, \tilde{J}_1,\tilde{J}_2, \dots, I_{n-k})$ where $\tilde{J}_1=(r_{p-1}+1,r_{p-1}+2, \dots, r_{p-1}+t-1)$ and $ \tilde{J}_2=(r_{p-1}+t,r_{p-1}+t+1,\dots, r_p)$.

\begin{figure}[!ht]
	\centering
	\begin{tikzpicture}
	\sffamily
	\node (a0) at (0,0) {$\sigma=(I_1,I_2,..I_{n-k})$};
	\node [below left=1cm of a0] (a1)  {$\tau=(I_1,I_2\dots J_1,J_2..I_{n-k})$};
	\node [below right=1cm of a0] (a2) {$\tilde{\tau}=(I_1,I_2 \dots \tilde{J}_1,\tilde{J}_2 \dots I_{n-k})$};
	\draw [blue, thick] (a0) -- (a1);
	\draw [blue, thick] (a0) -- (a2);
	\node[below=2cm of a0] (b) {$\omega=(I_1,I_2 \dots \tilde{J}_1,\{r_{p-1}+t+1\},J_2 \dots I_{n-k})$};
	\draw [blue, thick] (a1) -- (b);
	\draw [blue, thick] (a2) -- (b);
	
	\end{tikzpicture}
	\caption{}
\end{figure}

By induction hypothesis and step 1, we know $[\sigma:\tilde{\tau}], [\tau:\omega], [\tilde{\tau}: \omega ]$ and we can compute the $[\sigma:\tau]$ by using the fact the square of the boundary map vanishes in cellular chain complex.

This should be equal to the value defined above in definition because we have already showed $\partial^2=0$. 
\end{proof}

Here are some examples that illustrate the proof above:
\begin{example}
	Let $\alpha =(1,23,45,6)$ , $\beta= (1,2,3,45,6)$. The neighbors of $\mathrm{PV}(\alpha)$ are ordered as follows
	$v_1=(1,3,2,4,5,6)$,\\
	$v_2=(1,2,3,5,4,6)$,\\
	Since $\mathrm{PV}(\alpha)=\mathrm{PV}(\beta)$,  $g_{\alpha,\beta}=Id$. The vertex $\mathrm{PV}(\beta)$ has only one neighbor which is $(1,2,3,5,4,6)$ and hence the missing vertex is $v_1$. This shows that $\langle \partial \alpha,\beta \rangle=-1$
\end{example}

\begin{example}
	Let $\alpha =(1,23,45,6)$, $\beta_1= (1,2,3,45,6)$, $\beta_2= (1,23,4,5,6)$ and\\ $\gamma=(1,2,3,4,5,6)$. Observe	$\gamma <\beta_1, \gamma < \beta_2$ and $\beta_1, \beta_2< \alpha$.
	Clearly $\langle \partial \alpha, \beta_1 \rangle=-1 , \langle \partial \alpha, \beta_2\rangle=1$ and $\langle \partial \beta_1, \gamma\rangle=-1, \langle \partial \beta_2, \gamma \rangle=-1$. Thus we have $\langle \partial \alpha, \beta_1 \rangle \cdot \langle \partial \beta_1, \gamma \rangle=1$ and $\langle \partial \alpha, \beta_2 \rangle \cdot \langle \partial \beta_2, \gamma \rangle =-1$ showing that $\langle \partial^2 \alpha, \gamma \rangle =0$.
	
\end{example}

\section{A discrete Morse function for the cyclopermutahedron} \label{sec4}

In this section we recall the homology computations of $\CP_{n+1}$ done in \cite{nekrasov2016cyclopermutohedron}; most results in this section are not new.  
However, we give a slightly different proof of vanishing of boundary homomorphisms in the Morse complex of $\CP_{n+1}$.
The main ingredient is a technical criterion for gradient paths; it describes conditions under which two paths induce opposite orientations (see \Cref{good path lemma}, \emph{the good path lemma}).

We begin by recalling the discrete Morse function on $\CP_{n+1}$ that was introduced in \cite{nekrasov2016cyclopermutohedron}. 

\noindent\textbf{Step 1}. We pair together two cells
$$\alpha =(\dots {1}, I \dots)\ \mathrm{and}\ \beta =( \dots {1} \cup I \dots)$$
if $n + 1 \not\in I$.

We proceed for all $2 \leq k < n$, assuming that the $k$-th step is:

\textbf{Step} $\mathbf{k}$. We pair together two cells
$$\alpha =(\dots {k}, I \dots)\ \mathrm{and}\ \beta =(\dots {k}\cup I \dots)$$
if the following holds:
\begin{enumerate}
    \item $\alpha$ and $\beta$ were not paired at any of the previous steps.
    \item $n + 1 \not \in I$.
    \item $k < I$.
\end{enumerate}

\begin{example}
The cell $(2,43,1,56)$ is paired with the cell
$(243,1,56)$ on the second step. The cell $(4,5,3,1,26)$ is paired with the cell $(45,3,1,26)$ on the fourth step. The cell $(4,3,2,1,56)$ is not paired.
\end{example}

\begin{lemma}[{\cite[Lemma 4]{nekrasov2016cyclopermutohedron}}]
The above pairing is a discrete Morse function.
\end{lemma}

\begin{lemma}[{\cite[Lemma 5]{nekrasov2016cyclopermutohedron}}]
The critical cells of the above defined Morse function are
exactly all the cells of the following two types:\\
\textbf{Type 1.} Cells labeled by $( \nabla ,\{n + 1,\dots\} )$, where $\nabla$ is a string of singletons coming in decreasing order.\\
\textbf{Type 2.} Cells labeled by $( i, I, \{n + 1,\dots\} )$, where $i < I$.
\end{lemma}

\begin{lemma}[{\cite[Lemma 9]{nekrasov2016cyclopermutohedron}}] \label{paths between critical cells}
The following three cases describe all the gradient paths between
critical cells:
\begin{enumerate}
    \item Let $\beta = (\nabla^{'}, \{n+1,\dots\})$ and $\alpha = (\nabla, \{n+1, \dots \})$ be two cells of type 1. Then there are two gradient paths from $\beta$ to $\alpha$ iff $\nabla^{'}=\nabla \cup$ k for some k.
    \item Let $\beta = (\{i\}, \{j, k\}, \{n + 1,\dots\})$ and $\alpha = (\{k\}, \{j\}, \{i\}, \{n + 1,\dots\})$ be cells of type 2 and 1 respectively. Then there are two gradient paths from $\beta$ to $\alpha$.
    \item Let $\beta = (\{i\}, \{j\}, \{n + 1,\dots\})$ and $\alpha = (\nabla, \{n + 1, \dots\})$ be cell of type 2 and type 1 respectively. Then there are two gradient paths from $\beta$ to $\alpha$ iff $\nabla$ consists of three singletons, two of which are $\{i\}$ and $\{j\}$.
\end{enumerate}
\end{lemma}

Now we present the part which not only gives the different proof of the fact that boundary homomorphisms in the Morse complex vanish but will also help us in the next Section where the boundary homomorphisms are not zero. 

\begin{lemma}[The good path lemma]\label{good path lemma}
	Let $(t_1, t_2, s)$ be a triple such that $t_1 = (X,\{k\},I,Y)$ and $t_2 = (X,I,\{k\},Y)$ are two (p-1)-cells in the boundary of the p-cell $s = (X,\{k\}\cup I, Y)$ in $\CP_{n+1}$, where $X$ and $Y$ are sequences
	of sets. Then 
	\begin{align}\label{good}
	[t_1 : s][t_2 : s] = -1.
	\end{align}
\end{lemma}

\begin{proof}
Let $I = \{i_1,\dots , i_s\}$ with $i_1 < \dots< i_r < k < i_{r+1} < \dots< i_s$. Then,
\begin{align*}
v_0 &= PR(s) = (\Delta_X, i_1,\dots,i_r, k, i_{r+1},\dots, i_s,\Delta_Y). \\
v_{\tau_1} &= PR(\tau_1) = (\Delta_X, k, i_1,\dots, i_s,\Delta_Y).\\
v_{\tau_2} &= PR(\tau_2) = (\Delta_X, i_1,\dots,i_s, k,\Delta_Y).
\end{align*}

In the above expressions, given any sequence $Z$ of sets, $\Delta_Z$ denotes the partition of $Z$ into singletons (such that a sequence of singletons arising from a single set is contiguous and in ascending order). 
Thus $g_{\tau_1} = (k, i_1,\dots,i_r)$ and $g_{\tau_2} = (k, i_s,\dots,i_{r+1})$. Now, if $X = A_1A_2\dots A_a$, then denote by $\Vert X\Vert$ the quantity $\sum_{i=1}^a (\Vert{A_i}\Vert -1)$. Then $i_{\tau_1} = \Vert{X}\Vert + 1$ and $i_{\tau_2} = \Vert{X}\Vert + s$. Now we obtain
\begin{align*}
[\tau_1 : s]& = (\mathrm{sign} g_{\tau_1}) \cdot (-1)^{p+i_{\tau_1}}\\
&= (-1)^r \cdot (-1)^{p+\Vert{X}\Vert+1}\\
&= (-1)^{p+\Vert{X}\Vert+r+1}, and\\
[\tau_2 : s]& = (\mathrm{sign} g_{\tau_2}) \cdot (-1)^{p+i_{\tau_2}}\\
&= (-1)^{s-r} \cdot (-1)^{p+\Vert{X}\Vert+s}\\
&= (-1)^{p+\Vert{X}\Vert-r}
\end{align*}
and the result follows.
\end{proof}

\begin{definition}
Suppose we have a path $b_0, a_1, b_1, \dots , a_t, b_t, a_{t+1}$, where each triple $(a_i, a_{i+1}, b_i)$ is as above for $1 \leq i \leq t$. We call such a path a \emph{good path}.
\end{definition}

The following lemma follows directly from the proof of \cref{paths between critical cells}.

\begin{lemma}\label{paths are good}
    The gradient paths between critical cells in \cref{paths between critical cells} are good paths.
\end{lemma}

The above results lead to a rather simple proof for the vanishing of boundary maps in the Morse complex, which clearly mean that the homology groups of $\CP_{n+1}$ are torsion free and the Betti numbers are exactly equal to number of critical cells.

\begin{theorem}[{\cite[Lemma 7]{nekrasov2016cyclopermutohedron}}]\label{boundary vanish}
    The boundary operators of the Morse complex vanish.
\end{theorem}

\begin{proof}
From \cref{bdryeq2}, recall that
$$ \langle \tilde{\partial}\sigma, \tau \rangle = \sum_{\tilde{\sigma}_p < \sigma} \langle \partial \sigma, \tilde{\sigma} \rangle \sum_{c \in \Gamma(\tilde{\sigma}, \tau)} w(c)$$
Where, $\Gamma(\tilde{\sigma}, \tau)$ denote the set of all gradient paths from $\tilde{\sigma}$ to $\tau$.

By \cref{paths are good} the paths between critical cells are good paths. 
Hence $\forall \tilde{\sigma} < \sigma,\ \forall c \in \Gamma(\tilde{\sigma}, \tau)$, we have $ w(c)=1 $. 
Let us denote these two paths as 
$C:=\sigma, a_1, b_1, \dots, a_t, b_t,\tau$ and $D=\sigma, \alpha_1, \beta_1, \dots, \alpha_t, \beta_t,\tau$. 
Then $\langle \tilde{\partial}\sigma, \tau \rangle = \langle \partial \sigma, a_1 \rangle + \langle \partial \sigma, \alpha_1 \rangle$. Since the triple $(a_1, \sigma, \alpha_1)$ also forms a good pair, $[\sigma:a_1]\cdot [\sigma:\alpha_1]=-1$ implying $\tilde{\partial}=0$.
\end{proof}

\begin{corollary}[{\cite[Lemma 6]{nekrasov2016cyclopermutohedron}}]

The homology of $\CP_{n+1}$ is torsion free and the Betti numbers are given as follows
\[ b_i =
	\begin{cases} 
	2^n + \frac{2^n-3n-2}{2}, & i=n-2; \\
	\binom{n}{i}, & 0 \leq i < n-2.
	\end{cases}
\]
\end{corollary}


\section{The bi-cyclopermutohedron and its mod \texorpdfstring{$2$}{2} homology}\label{sec5}

In this Section we first construct a certain quotient of $\CP_{n+1}$ then define a discrete Morse function on it and use it compute the mod-$2$ homology. 

Consider a $\mathbb{Z}_2$ action on $\CP_{n+1}$ given by the involution.
\begin{equation}\label{involution}
\begin{gathered}
r: \CP_{n+1} \longrightarrow CP_{n+1}\\
(I_1,I_2,\dots,I_{k-1},I_k) \mapsto (I_{k-1},\dots, I_2,I_1,I_k).
\end{gathered}
\end{equation}
Essentially the action identifies cyclically ordered partitions that are obtained by cyclically permuting blocks in either direction.
Clearly this action is fixed point free and we have the quotient $\CP_{n+1}/\mathbb{Z}_2$ which we name the \emph{bi-cyclopermutohedron} and denote it by $\QP_{n+1}$. 
See \Cref{qpexample} for an example when $n=3$.

\begin{definition}
    The regular CW complex bi-cyclopermutohedron $\QP_{n+1}$ is defined as:
\begin{itemize}
	\item For $k = 0, 1 ,\dots, n-2$, the k-cells of $\QP_{n+1}$ are labeled by (all possible) bi-cyclically ordered partitions of the set $[n + 1]$ into $(n-k+1)$ non-empty parts.
	\item A closed cell $\bar{F}$ contains a cell $F^{\prime}$ whenever the label of $F^{\prime}$ refines that of $F$.
\end{itemize}
\end{definition}

Recall that for a generic length vector $\ell:=(l_0,l_1, \dots, l_{n}) \in \mathbb{R}_+^{n+1}$ the moduli space of planar polygons $\mathscr{M}_\ell$ admits a natural free $\mathbb{Z}_2$ action; wherein each polygon is mapped to its reflection about the X-axis. 
The quotient space $\mathscr{M}_\ell/\mathbb{Z}_2$, denoted $\overline{\mathscr{M}}_\ell$, is the space of polygons viewed up to the action of all isometries.
Note that the involution defined in \Cref{involution} mimics the above reflection. 
Moreover the complex $\QP_{n+1}$ is the ``universal object'' for the moduli space $\overline{\mathscr{M}}_\ell$ in the same sense as  $\CP_{n+1}$ is for $\mathscr{M}_\ell$.

\begin{figure}
\centering
\definecolor{zzttqq}{rgb}{0.6,0.2,0.}
\definecolor{ududff}{rgb}{0.30196078431372547,0.30196078431372547,1.}
\begin{tikzpicture}[line cap=round,line join=round,>=triangle 45,x=1.0cm,y=1.0cm]
\clip(1.5,0.) rectangle (9.5,7.);
\draw[line width=1.pt,color=zzttqq] (4.,2.) -- (5.58,4.48) -- (7.,2.) -- cycle;
\draw[line width=1.pt,color=zzttqq] (4,2) to[out=90,in=200] (5.58,4.48);
\draw[line width=1.pt,color=zzttqq] (5.58,4.48) to[out=-20,in=-270] (7,2);
\draw[line width=1.pt,color=zzttqq] (4,2) to[out=-30,in=-140] (7,2);
\draw [line width=1.pt,color=zzttqq] (4.,2.)-- (5.58,4.48);
\draw [line width=1.pt,color=zzttqq] (5.58,4.48)-- (7.,2.);
\draw [line width=1.pt,color=zzttqq] (7.,2.)-- (4.,2.);
\draw [line width=1.pt] (5.56,5.56) circle (0.35cm);
\draw [line width=1.pt] (8.,1.8) circle (0.35cm);
\draw [line width=1.pt] (2.8,1.72) circle (0.35cm);
\draw (5.1,5.3) node[anchor=north west] {\{2\}};
\draw (7.58,1.5) node[anchor=north west] {\{2\}};
\draw (1.7,2) node[anchor=north west] {\{2\}};
\draw (7.6,2.7) node[anchor=north west] {\{1\}};
\draw (4.4,5.9) node[anchor=north west] {\{1\}};
\draw (2.3,1.4) node[anchor=north west] {\{1\}};
\draw (5.1,6.5) node[anchor=north west] {\{3\}};
\draw (6.9,2.1) node[anchor=north west] {\{3\}};
\draw (2.3,2.6) node[anchor=north west] {\{3\}};
\draw (8.22,2.1) node[anchor=north west] {\{4\}};
\draw (5.76,5.8) node[anchor=north west] {\{4\}};
\draw (3.1,1.96) node[anchor=north west] {\{4\}};
\begin{scriptsize}
\draw [fill=ududff] (4.,2.) circle (2.5pt);
\draw [fill=ududff] (5.58,4.48) circle (2.5pt);
\draw [fill=ududff] (7.,2.) circle (2.5pt);
\end{scriptsize}
\end{tikzpicture}
  \caption{The complex $\QP_4$}
  \label{qpexample}
\end{figure}
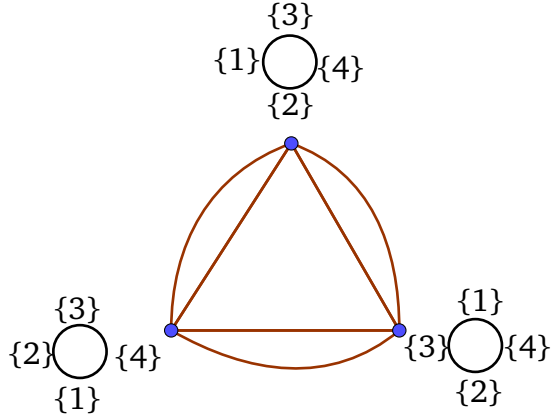
    
We begin by introducing some notions that are useful when dealing with equivalence classes of bi-cyclically ordered partitions. 
The aim is to show how to choose a nice representative for these equivalence classes. 
These ideas were originally introduced by Adhikari in his Masters' thesis \cite{nachiketa2017} written under the supervision of the first author.

\begin{definition}
Let $\lambda = (I_1,I_2, \ldots, I_k)$ be a cyclically ordered partition of $n$. Let $j$ be the greatest element outside the $n+1$ set and $j \in I_l$ for $1\leq l < k$. Further, let $i$ be the greatest element outside $n+1$ set and $I_l$ and $i \in I_m$ for $1 \leq m < k$ and $m \neq l$. Then $\lambda$ is said to be of class $(i,j)$ if $m \leq l$ and of class $(j,i)$ otherwise. The class of $\lambda$ is denoted cl($\lambda$).
\end{definition}

\begin{definition}
	Let $\lambda$ be a cell of cl$(i,j)$. $\lambda$ is called an \emph{ ascending cell} if $i < j$ and \emph{ descending} otherwise.
\end{definition}

\begin{lemma}
 The cells of $\CP_{n+1}$ can be partitioned into two classes: one with ascending cells and the other with descending cells. Involution defined in Equation \eqref{involution} establishes bijection between these two classes.
	\begin{center}
		\{Ascending Cells\} $\longleftrightarrow$\{Descending Cells\}
	\end{center}
\end{lemma}

Therefore, each cell in the quotient complex $\QP_{n+1}$ is an equivalence class containing an ascending cell and a descending cell (each is the reflection of the other). 
Let $\pi:\CP_{n+1} \rightarrow \QP_{n+1}$ be the quotient map and for $\lambda\in\CP_{n+1}$ we denote $\pi(\lambda)$ by $ \bar{\lambda}$. 
Henceforth, unless otherwise mentioned, for every cell $\bar{\lambda}\in\QP_{n+1}$ we assume that $\lambda$, the chosen representative, is ascending.

\begin{definition}
A cell $\bar{\lambda}$ is said to be of \emph{class} $\{i,j\}$ if one of the preimages under the quotient map is of class $(i,j)$. We denote the class by cl($\bar{\lambda}$).
\end{definition}

Now, we define an order on the cells of $\QP_{n+1}$.

\begin{definition}
Let $\bar{\lambda}$ and $\bar{\lambda}^{\prime}$ be two cells of class $\{i,j\}$ and $\{i^{\prime},j^{\prime}\}$ respectively. If min\{$i,j$\} $\leq$ min\{$i^{\prime},j^{\prime}$\} and max\{$i,j$\} $\leq$ max\{$i^{\prime},j^{\prime}$\}, then $\bar{\lambda}^{\prime}$ is said to be \emph{ higher} than $\bar{\lambda}$.
\end{definition}

\begin{lemma}\label{boundary}
	If $\bar{\alpha}, \bar{\beta} \in \QP_{n+1}$ and $\bar{\alpha}$ is contained in the boundary of $\bar{\beta}$, then $\bar{\alpha}$ is higher than $\bar{\beta}$.
\end{lemma}
\begin{proof}
Let cl$(\bar{\beta})=\{i,j\}$ and $i<j$. Since $\bar{\alpha}$ is in the boundary of $\bar{\beta}$, each block of $\bar{\alpha}$ is a subset of a block of $\bar{\beta}$. Therefore, the largest element in $\bar{\alpha}$ outside the $n+1$ set (say $j^\prime$) has to be greater than or equal to $j$, \textit{i.e.,} $j^\prime\geq j$. Similarly, the second largest element outside the $n+1$ set and the set containing $j^\prime$ has to be greater than or equal to $i$.
\end{proof}



Let us define a discrete Morse function on $\QP_{n+1}$ inductively.

\noindent\textbf{Step $1$:} Pair $\alpha=(\ldots,1,I,\ldots) \text{ and } \beta=(\ldots,1\cup I,\ldots)$ in $\CP_{n+1}$ if the following conditions hold:
\begin{enumerate}
	\item $n+1 \not\in I$.
	\item $\alpha$ is ascending.
	\item cl($\alpha$) = cl($\beta$).
\end{enumerate}
Note that, the conditions $2$ and $3$ together imply that $\beta$ also is an ascending cell.

\noindent\textbf{Step $k$:} Pair $\alpha=(\ldots,k,I,\ldots)$ and $\beta=(\ldots,k\cup I,\ldots)$ if the following conditions hold.
\begin{enumerate}
	\item $n+1 \not\in I$.
	\item $\alpha$ and $\beta$ have not yet been paired.
	\item $\alpha$ is ascending.
	\item cl($\alpha$) = cl($\beta$)
\end{enumerate}

After the $(n-2)^\text{nd}$ step, we have-

\noindent\textbf{The final step:} If $\alpha$ and $\beta$ have been paired in $\CP_{n+1}$, then match $\bar{\alpha}$ with $\bar{\beta}$ in $\QP_{n+1}$ (here $\bar{\alpha}$ with $\bar{\beta}$ represents the image of $\alpha$ and $\beta$ under the map $\pi:\CP_{n+1} \rightarrow \QP_{n+1}$).


\begin{lemma}\label{higherpath}
	If there is a gradient path $$\bar{\beta}_0,\bar{\alpha}_1,\bar{\beta}_1,\dots, \bar{\alpha}_p,$$ then $\bar{\alpha_p}$ is higher than $\bar{\beta_0}$.
\end{lemma}
\begin{proof}
Since we only match the cells in the same class, we have $\text{cl}(\bar{\alpha}_i)=\text{cl}(\bar{\beta}_i)$ for each $i \in \{1,2,\dots, p-1\}$. Moreover, using \cref{boundary}, we get that $\text{cl}(\bar{\alpha}_j)\geq\text{cl}(\bar{\beta}_{j-1})$ for each $j \in \{1,\dots, p\}$. Thus, the result follows.
\end{proof}

\begin{theorem}
The pairing on $\QP_{n+1}$, as described above is a discrete Morse function.
\end{theorem}
\begin{proof}
On the contrary, assume that the matching defined is not acyclic, \emph{i.e.} there is a path
$$\bar{\alpha}_0,\bar{\beta}_0,\bar{\alpha}_1,\bar{\beta}_1,\dots,\bar{\alpha}_p$$ with $p>1$ and $\bar{\alpha}_0=\bar{\alpha}_p$. Since $\bar{\alpha}_0$ and $\bar{\beta}_0$ are matched, they are in the same class, Therefore, using \cref{higherpath}, we get that $\text{cl}(\bar{\alpha}_0)=\text{cl}(\bar{\beta}_0)=\text{cl}(\bar{\alpha}_1)=\ldots = \text{cl}(\bar{\alpha}_p)$.

We now “lift” this cycle to $\CP_{n+1}$. Let $a_0$ be the ascending cell such that $\pi(a_0) = \bar{\alpha_0}$. Let $b_0$ the cell with which $a_0$ is paired (in particular, $\pi(b_0) = \bar{\beta_0}$). Next, suppose $a_1$ is ascending with $\pi(a_1) = \bar{\alpha_1}$. Note that $cl(a_0) = cl(b_0) = cl(a_1) = (i, j)$ for some $i < j$. If $a_1$ is not in the boundary of $b_0$, then it must be in the boundary of $r(b_0)$ (for otherwise $a_1$ would not be in the boundary of $b_0$). But since $cl(r(b_0)) = (j, i)$, we have a cell of class $(i, j)$ in the boundary of a cell of class $(j, i)$, which is impossible. Hence $a_1$ is in the boundary of $b_0$. Continuing thus, we obtain a path $a_0, b_0, a_1, b_1, \dots , a_p$
with $a_i$ and $b_i$ ascending for each $i$ (and, in particular, $a_0 = a_p$). Thus the cycle in $\QP_{n+1}$ lifts to the cycle in $\CP_{n+1}$. The matching on the ascending cells is, however, a subset of the matching of $\CP_{n+1}$ described in \cref{sec4}, and hence the cycle cannot exist.
\end{proof}

\textbf{Notation:} Let $\lambda$ denote the unique ascending representative of $\bar{\lambda} \in \QP_{n+1}$.

\begin{theorem}
	The critical cells of the discrete Morse function on $\QP_{n+1}$ are the images under $\pi$ of the cells of the type (i,I,$\nabla$,N) with $\nabla <i< I$. 
\end{theorem}

\begin{proof}
	Assume $\bar{\lambda}=(I_1,I_2 ,\dots,  I_k)$ is critical and cl($\bar{\lambda}$)=$\{i,j\}$.
	
\textbf{Claim 1: } $ i \in I_1$.

\emph{Proof.}
Assume, without loss of generality that $i \in I_2$.
\begin{enumerate}
	\item If $|I_1| = 1$, then by construction $I_{1} < i \leq I_2$. Hence, the cell $ \lambda =(I_1,I_2 ,\dots,  I_k)$ can be matched with $ \lambda' =(I_1 \cup I_2 ,\dots, I_k)$ as they have the same class type.
			
	\item Let $|I_1| > 1$ and denote the minimum of $I_{1}$ by $m$. Then the cells $ \lambda =(I_1,I_2 ,\dots,  I_k)$ and $ \lambda' =(m,I_1-\{m\},I_2 ,\dots, I_k)$ can be matched as they have the same class type.
\end{enumerate}
		
\textbf{Claim 2: }$I_1=\{i\}$

\emph{Proof.}
Assume on the contrary that $|I_1|>1$. Denote the minimum of $I_1$ be $m$. Then the cells $ \lambda$ and $ \lambda' =(m,I_{1}-\{m\} ,\dots,  I_k)$ can be matched.
	    
\textbf{Claim 3: }$j \in I_2$.

\emph{Proof.}
Assume without loss of generality that $j \in I_3$.
\begin{enumerate}
	\item $|I_{2}| = 1$.
	\begin{enumerate}
		\item If $I_{2} < I_3$, then the cell $ \lambda =(\{i\},I_2 ,\dots,  I_k)$ and $ \lambda' =(\{i\},I_2\cup I_3 ,\dots, I_k)$ can be matched as they have the same class type.
				
		\item If $I_{2} \not < I_3$, then $\exists m \in I_{3}$ such that $m<I_{2}$ The cell $ \lambda =(\{i\},I_2 ,\dots,  I_k)$ and $ \lambda' =(\{i\},I_2, m, I_3-\{m\} ,\dots, I_k)$ can be matched as they have the same class type.
	\end{enumerate} 
			
	\item Let $|I_2| > 1$ and denote the minimum of $I_{2}$ by m. Then the cells $ \lambda =(\{i\},I_2 ,\dots,  I_k)$ and $ \lambda' =(\{i\},m,I_2-\{m\},\dots, I_k)$ can be matched as they have the same class type.	
\end{enumerate}

\textbf{Claim 4: }$i<I_2$.

\emph{Proof.}
If $\exists m \in I_2$ such that $m<i$ then Then the cells $ \lambda =(\{i\},I_2 ,\dots,  I_k)$ and $ \lambda' =(\{i\},m,I_{2}-\{m\} ,\dots,  I_k)$ can be matched as they have the same class type.
	    
A similar argument shows that all other subsets $I_3 ,\dots,  I_{k-1}$ are singletons arranged in decreasing order.
\end{proof}
\begin{proposition}
	Let $\bar{\alpha} = (i,I,\nabla,N)$ and $\bar{\beta} = (j,J,\nabla', N')$. If there is a path from $\bar{\alpha}$ to $\bar{\beta}$ then either $ N \cap J=\emptyset$ or $I \cap J=\emptyset$.
\end{proposition}

\begin{proof}
	let $x\in N \cap J$ and $t\in I \cap J$. Clearly $x>j\geq i$ and $t>j\geq i$. Denote the maximum element of $I$ by $m$ and $m>j$. Let the path from $\bar{\alpha}$ to $\bar{\beta}$ be
	$$\bar{\alpha}=\bar{\alpha}_0,\bar{\beta}_0,\bar{\alpha}_1,\bar{\beta}_1,\dots, \bar{\alpha}_p,\bar{\beta}_p=\bar{\beta}$$ 
	During the course of the path, $x$ leaves the set $N$, say at $\bar{\beta}_{k}$. Then min cl$(\bar{\beta}_{k}) \geq \min\{x, m\} \geq \min\{x,t\} > j=\min \mathrm{cl}(\bar{\beta})$. This contradicts the fact that the class increases along a gradient path.
\end{proof}

The following theorem about the paths between critical cells is crucial in computing the homology of $\QP_{n+1}$.

\begin{theorem}\label{paths}
Let $\bar{\alpha} = (i,I,\nabla,N)$ and $\bar{\beta}=(j,J,\nabla', N')$ be two critical cells. If there is a path from $\bar{\alpha}$ to $\bar{\beta}$, then $\bar{\beta}$ takes exactly one of the following form.
	\begin{enumerate}
		\item \begin{enumerate}
			\item If $J=I$ and $N' =N-t$. Then \[ \bar\beta=
			\begin{cases} 
			(t,I,i, \nabla, N'), & \mathrm{if\ } i<t<I,\\
			(i, I, \nabla \cup t, N'), & \mathrm{if\ } t<i. \\
			\end{cases}
			\]
			
			\item  If $N'= N-t$ and $|I|=1$. Then \[ \bar\beta=
			\begin{cases} 
			(I,t,i, \nabla, N'), & \mathrm{if\ } t>I, \\
			(t,I,i, \nabla, N'), & \mathrm{if\ } i<t<I, \\
			(i,I,t\cup \nabla, N'), & \mathrm{if\ } t<i.	
			\end{cases}
			\]
		\end{enumerate}
	
		\item If $N=N'$ and $J=I-j$. Then $$\bar\beta = (j,J,i,\nabla,N).$$

	\end{enumerate}
\end{theorem}

\begin{proof}
	We will prove this explicitly \textit{i.e.,} by following the paths from $\bar{\alpha}$ to $\bar{\beta}$. Let $\nabla$ be $\{a_1\},\{a_2\},\dots, \{a_l\}$ with $a_1>a_2> \dots >a_l$
	\begin{enumerate}
		\item 
		\begin{enumerate}
			\item 
			\begin{enumerate}
			    \item For $\bar{\alpha} = (i,I,\nabla,N)$, $\beta =(t,I,i, \nabla, N')$ and $i<t<I$, we've the paths
		        \begin{equation}
		        \begin{gathered}
			    (i,I,a_1,a_2 ,\dots,  a_l,N)\\
			    (i,I,a_1,a_2 ,\dots,  a_l,t,N-t)\\
			    (t, a_l ,\dots,  a_1,I,i,N-t)\\
			    (t, a_l ,\dots,  a_1\cup I,i,N-t)\\
			    (t, a_l ,\dots,  I,a_1,i,N-t)\\
			    \vdots\\
			    (t,I,a_l ,\dots,  a_1,i,N-t)\\
			    (t,I,a_l ,\dots,  \{a_1,i\},N-t)\\
			    (t,I,a_l ,\dots,  i,a_1,N-t)\\
			    \vdots\\
			    (t,I,i,a_l,\dots,  a_1,N-t)\\
			    (t,I,i,a_l ,\dots,  \{a_2,a_1\},N-t)\\
			    (t,I,i,a_l ,\dots,  a_0,a_1,N-t)\\
			    \vdots\\
			    (t,I,i,a_1,a_2 ,\dots,  a_l,N-t)
		        \end{gathered}
		        \end{equation}
		and 
		        \begin{equation}
		        \begin{gathered}	
			    (i,I,a_1,a_2 ,\dots,  a_l,N)\\
			    (i,I,a_1,a_2 ,\dots,  a_l,N-t,t)\\
			    (t,i\cup I,a_1,a_2 ,\dots,  a_l,N-t)\\
			    (t,I,i,a_1,a_2 ,\dots,  a_l,N-t).
		        \end{gathered}
		        \end{equation}
		
				\item For $\bar{\alpha} = (i,I,\nabla,N)$, $\beta =(i, I, \nabla \cup t, N')$ and $a_p>t>a_{p+1}$, we've the paths
			    \begin{equation}
			    \begin{gathered}
				(i,I,a_1,a_2 ,\dots,  a_l,N)\\
				(i,I,a_1,a_2 ,\dots,  a_l,t,N-t)\\
				(i,I,a_1,a_2 ,\dots,  \{a_l,t\},N-t)\\
				(i,I,a_1,a_2 ,\dots,   t,a_l,N-t)\\
				(i,I,a_1,a_2 ,\dots,  \{a_{l-1},t\},a_l,N-t)\\
				\vdots\\
				(i,I,a_1,a_2 ,\dots,  \{a_{p+1},t\} ,\dots,  a_l,N-t)\\
				(i,I,a_1,a_2 ,\dots,  t,a_{p+1} ,\dots,  a_l,N-t)
			    \end{gathered}
			    \end{equation}
			    and
			    \begin{equation}
			    \begin{gathered}	
				(i,I,a_1,a_2 ,\dots,  a_l,N)\\
				(i,I,a_1,a_2 ,\dots,  a_l,N-t,t)\\
				(\{t,i\},I,a_1,a_2 ,\dots,  a_l,N-t)\\
				(i,t,I,a_1,a_2 ,\dots,  a_l,N-t)\\
				(i,t\cup I,a_1,a_2 ,\dots,  a_l,N-t)\\
				(i,I,t,a_1,a_2 ,\dots,  a_l,N-t)\\
				(i,I,\{t,a_1\},a_2 ,\dots,  a_l,N-t)\\
				(i,I,a_1,t,a_2 ,\dots,  a_l,N-t)\\
				\vdots\\
				(i,I,a_1,a_2 ,\dots,  \{t,a_p\} ,\dots,  a_l,N-t)\\
				(i,I,a_1,a_2 ,\dots,  a_p,t ,\dots,  a_l,N-t).	
			    \end{gathered}
			    \end{equation}
			\end{enumerate}

		\item For $\bar{\alpha} = (i,I,\nabla,N)$, $\beta =(I,t,i, \nabla, N')$ and $t>I$, we have  paths
		
		\begin{equation}
		\begin{gathered}
		(i,I,a_1,a_2 ,\dots,  a_l,N)\\
		(i,I,a_1,a_2 ,\dots,  a_l,N-t,t)\\
		(t,i,I,a_1,a_2 ,\dots,  a_l,N-t)\\
		(a_l,a_{l-1} ,\dots,  a_1,I,i,t,N-t)\\
		(a_l,a_{l-1} ,\dots,  a_1,I,i\cup t,N-t)\\
		(a_l,a_{l-1} ,\dots,  a_1,I,t,i,N-t)\\
		(a_l,a_{l-1} ,\dots,  a_1\cup I,t,i,N-t)\\
		(a_l,a_{l-1} ,\dots,  I,a_1, t,i,N-t)\\
		\vdots\\
		(I,a_l,a_{l-1} ,\dots,  a_1, t,i,N-t)\\
		(I,a_l,a_{l-1} ,\dots,  \{a_2,a_1\}, t,i,N-t)\\
		(I,a_l,a_{l-1} ,\dots,  a_1,a_2 t,i,N-t)\\
		\vdots\\
		(t,I,i, a_1,a_2 ,\dots,  a_l, N-t)
		\end{gathered}
		\end{equation}
		and
		
		\begin{equation}
		\begin{gathered}	
		(i,I,a_1,a_2 ,\dots,  a_l,N)\\
		(i,I,a_1,a_2 ,\dots,  a_l,t,N-\{t\})\\
		(i\cup I,a_1,a_2 ,\dots,  a_l,t,N-\{t\})\\
		(I,i,a_1,a_2 ,\dots,  a_l,t,N-\{t\})\\
		\vdots\\
		(I,t,i,a_1,a_2 ,\dots,  a_l,N-\{t\}).
		\end{gathered}
		\end{equation}
		\end{enumerate}
		
		The proofs for the other two cases are similar to (a).
		
		\item For $\bar{\alpha} = (i,I,\nabla,N)$, $\beta =(j,I-\{j\},i, \nabla, N)$ and $i<j$, we've the paths
		\begin{equation}
		\begin{gathered}
			(i,I,a_1,a_2 ,\dots,  a_l,N)\\
			(i,j,I-j,a_1,a_2 ,\dots,  a_l,N)\\
			(\{i,j\}, I-j,a_1,a_2 ,\dots,  a_l,N)\\
			(j,i, I-j,a_1,a_2 ,\dots,  a_l,N)\\
			(j,i \cup I-j,a_1,a_2 ,\dots,  a_l,N)\\
			(j,I-j,i,a_1,a_2 ,\dots,  a_l,N)
		\end{gathered}
		\end{equation}
		and
		\begin{equation}
		\begin{gathered}	
			(i,I,a_1,a_2 ,\dots,  a_l,N)\\
			(i,I-j,j,a_1,a_2 ,\dots,  a_l,N)\\
			(a_l,a_{l-1} ,\dots,  a_1, j,I-j,i,N)\\
			(a_l,a_{l-1} ,\dots,  \{a_1, j\},I-j,i,N)\\
			(a_l,a_{l-1} ,\dots,  j,a_1,I-j,i,N)\\
			\vdots\\
			(j,a_l,a_{l-1} ,\dots,  a_1,I-j,i,N)\\
			(j,a_l,a_{l-1} ,\dots,  a_1 \cup I-j,i,N)\\
			(j,a_l,a_{l-1} ,\dots,  I-j,a_1,i,N)\\
			\vdots\\
			(j,I-j,i,a_1,a_2 ,\dots,  a_l, N).	
		\end{gathered}
		\end{equation}

	\end{enumerate}
\end{proof}

The $\mathbb{Z}_2$-homology of $\QP_{n+1}$ can be computed directly from \cref{paths}.

\begin{theorem}\label{mod2thm}
The $\mathbb{Z}_2$-homology of $\QP_{n+1}$ is given as follows
\[ H_i(\QP_{n+1}, \mathbb{Z}_2) =
	\begin{cases}
	\displaystyle \bigoplus_{\xi(n,i)}\mathbb{Z}_2, & 0 \leq i \leq n-2;\\
	0, & \mathrm{otherwise}.
	\end{cases}
\]
Where, $\xi(n,i)$ denotes the sum $\displaystyle \sum_{k=0}^{i} \binom{n}{k}$.
\end{theorem}

\begin{proof}
One can infer from \cref{paths} that between any two critical cells in consecutive dimensions either there  is no path between them or there are exactly two paths.  
This implies that the boundary maps in the Morse complex of $\QP_{n+1}$  with $\mathbb{Z}_2$- coefficients are zero. 
So, the mod-$2$ Betti numbers are given by the number of critical cells. 
Once the dimension is fixed, say $i$, the $(n+1)$-set completely determines the critical cell and it contains at most $i+1$ elements.
\end{proof}

\section{The integral homology of \texorpdfstring{$\QP_{n+1}$}{cpn}}\label{sec6}

To compute the $\mathbb{Z}$-homology we need a well-defined notion of orientation on the cells of $\QP_{n+1}$. So, we induce an orientation on each cell of $\QP_{n+1}$ from its ascending representative in $\CP_{n+1}$. But, this is not sufficient to compute the $\mathbb{Z}$-homology because the paths between some critical cells involve identification of ascending and descending cells. If $\{\sigma, \tau\}=\bar{\sigma} \in \QP_{n+1}$ and $\sigma$ ascending, a compatible way of inducing an orientation on the cell $\sigma$ from canonical orientation of $\tau$ is required and is defined as follows.

Let $\cL$ denote the ordered neighbors of the vertex $\mathrm{PV}(\tau)$ as defined in \cref{canonical orientation}. The ordered vertices $\cL'$ obtained from $\cL$ by the action of $r$ on individual elements induce a orientation on $\sigma$. 

Now, we need to compute the difference in the orientation induced by each representative on $\bar{\sigma}$. The following examples demonstrate the existence of a closed-expression for the difference in the induced orientations. 

\begin{example} 
	Let $n=8$, $\sigma=\{1,2,3\}\{4,5\}\{6,7,8,9\}$ and $\tau=\{4,5\}\{1,2,3\}\{6,7,8,9\}$ are two cells in $\CP_9$ such that $\bar{\sigma}=\bar{\tau}$ in $\QP_9$.
	
\begin{center}
	\begin{tabular}{ c c }
	Neighbors of $\mathrm{PV}(\sigma)$ & Neighbors of $\mathrm{PV}(\tau)$\\
	$v_0=(\{1\}\{2\}\{3\}\{4\}\{5\}\{6\}\{7\}\{8\}\{9\})$ & $w_0=(\{4\}\{5\}\{1\}\{2\}\{3\}\{6\}\{7\}\{8\}\{9\})$ \\ 
	$v_1=(\{2\}\{1\}\{3\}\{4\}\{5\}\{6\}\{7\}\{8\}\{9\})$ & $w_1=(\{5\}\{4\}\{1\}\{2\}\{3\}\{6\}\{7\}\{8\}\{9\})$ \\  
	$v_2=(\{1\}\{3\}\{2\}\{4\}\{5\}\{6\}\{7\}\{8\}\{9\})$ & $w_2=(\{4\}\{5\}\{2\}\{1\}\{3\}\{6\}\{7\}\{8\}\{9\})$ \\
	$v_3=(\{1\}\{2\}\{3\}\{5\}\{4\}\{6\}\{7\}\{8\}\{9\})$ & $w_3=(\{4\}\{5\}\{1\}\{3\}\{2\}\{6\}\{7\}\{8\}\{9\})$ \\
	$v_4=(\{1\}\{2\}\{3\}\{4\}\{5\}\{7\}\{6\}\{8\}\{9\})$ &  $w_4=(\{4\}\{5\}\{1\}\{2\}\{3\}\{7\}\{6\}\{8\}\{9\})$ \\
	$v_5=(\{1\}\{2\}\{3\}\{4\}\{5\}\{6\}\{8\}\{7\}\{9\})$ & $w_5=(\{4\}\{5\}\{1\}\{2\}\{3\}\{6\}\{8\}\{7\}\{9\})$ \\
	$v_6=(\{1\}\{2\}\{3\}\{4\}\{5\}\{6\}\{7\}\{9\}\{8\})$ & $w_6=(\{4\}\{5\}\{1\}\{2\}\{3\}\{6\}\{7\}\{9\}\{8\})$ \\
	\end{tabular}
\end{center}

Let $\xi$ be the permutation which takes the vertex $\mathrm{PV}(\sigma)$ to the vertex $r(\mathrm{PV}(\tau))$.
	
\begin{center}
	\begin{tabular}{ c c }
	$v'_0 = (\{8\}\{7\}\{6\}\{3\}\{2\}\{1\}\{5\}\{4\}\{9\})$ & $w'_0 = (\{8\}\{7\}\{6\}\{3\}\{2\}\{1\}\{5\}\{4\}\{9\})$ \\
	$v'_1=(\{8\}\{7\}\{6\}\{2\}\{3\}\{1\}\{5\}\{4\}\{9\})$ & $w'_1=(\{8\}\{7\}\{6\}\{3\}\{2\}\{1\}\{4\}\{5\}\{9\})$ \\
	$v'_2=(\{8\}\{7\}\{6\}\{3\}\{1\}\{2\}\{5\}\{4\}\{9\})$ & $w'_2=(\{8\}\{7\}\{6\}\{3\}\{1\}\{2\}\{5\}\{4\}\{9\})$ \\
	$v'_3=(\{8\}\{7\}\{6\}\{3\}\{2\}\{1\}\{4\}\{5\}\{9\})$ & $w'_3=(\{8\}\{7\}\{6\}\{2\}\{3\}\{1\}\{5\}\{4\}\{9\})$ \\
	$v'_4=(\{7\}\{6\}\{3\}\{2\}\{1\}\{5\}\{4\}\{8\}\{9\})$ & $w'_4=(\{8\}\{6\}\{7\}\{3\}\{2\}\{1\}\{5\}\{4\}\{9\})$ \\
	$v'_5=(\{7\}\{8\}\{6\}\{3\}\{2\}\{1\}\{5\}\{4\}\{9\})$ & $w'_5=(\{7\}\{8\}\{6\}\{3\}\{2\}\{1\}\{5\}\{4\}\{9\})$ \\
	$v'_6=(\{8\}\{6\}\{7\}\{3\}\{2\}\{1\}\{5\}\{4\}\{9\})$ & $w'_6=(\{7\}\{6\}\{3\}\{2\}\{1\}\{5\}\{4\}\{8\}\{9\})$
	\end{tabular}
\end{center}

Here $v'_i$(resp. $w'_i$) denote the images of $v_i$(resp. $w_i$) under the map $\xi$(resp. $r$).

\begin{figure}[H]
	\centering
	\begin{tikzpicture}[line width=1pt,>=latex]
	\sffamily
	\node (a1) {$\pi(v_0)$};
	\node[below=0.3cm of a1] (a2) {$\xi(v_1)$};
	\node[below=0.3cm of a2] (a3) {$\xi(v_2)$};
	\node[below=0.3cm of a3] (a4) {$\xi(v_3)$};
	\node[below=0.3cm of a4] (a5) {$\xi(v_4)$};
	\node[below=0.3cm of a5] (a6) {$\xi(v_5)$};
	\node[below=0.3cm of a6] (a7) {$\xi(v_6)$};
	
	\node[right=4cm of a1] (b1) {$r(w_0)$};
	\node[below=0.3cm of b1] (b2) {$r(w_1)$};
	\node[below=0.3cm of b2] (b3) {$r(w_2)$};
	\node[below=0.3cm of b3] (b4) {$r(w_3)$};
	\node[below=0.3cm of b4] (b5) {$r(w_4)$};
	\node[below=0.3cm of b5] (b6) {$r(w_5)$};
	\node[below=0.3cm of b6] (b7) {$r(w_6)$};
		
	\node[shape=ellipse,draw=gray,minimum size=3cm,fit={(a1) (a7)}] {};
	\node[shape=ellipse,draw=gray,minimum size=3cm,fit={(b1) (b7)}] {};

	\draw[->,gray] (a1) -- (b1);
	\draw[->,gray] (a2) -- (b4);
	\draw[->,gray] (a3) -- (b3);
	\draw[->,gray] (a4) -- (b2);
	\draw[->,gray] (a5) -- (b7);
	\draw[->,gray] (a6) -- (b6);
	\draw[->,gray] (a7) -- (b5);
		
	\end{tikzpicture}
	\caption{}
	\label{fig:comparison2}
\end{figure}

Comparing the orientations induced on the cell $\bar{\sigma}=\bar{\tau}$ by the cells $\sigma$ and $\tau$ involves exactly two permutations. The permutation $\xi$ and the permutation from the comparing the orientation induced on $\sigma \in \CP_9$ by the vertices $v_0$ and $\xi(w_0)$, refer \cref{fig:comparison2}.
	
\end{example}

From the example, it is clear that the permutation involved in comparing the orientations induced are of the type
$$ (1,2,3,\dots,k) \rightarrow (k,k-1,\dots,2,1).$$
and the following function is useful in computing the sign of such permutations.

\begin{definition}
	Define a function, $\mathrm{sgn}: [n] \cup \{0\}\rightarrow \{1,-1\} $ as follows, given $s \in [n]\cup\{0\}$
	\[ \mathrm{\mathrm{sgn}}(s)=
	\begin{cases} 
	(-1)^{\frac{s-1}{2}}, & \mathrm{if\ s\ is\ odd}, \\
	(-1)^{\frac{s}{2}}, & \mathrm{if\ s\ is\ even}. 	
	\end{cases}
	\]
\end{definition}

Let $\sigma= (I_1,I_2, \dots, I_k,I_{\omega})$ and $\tau= r(\sigma)=(I_k,I_{k-1}, \dots, I_1,I_{\omega})$ be two cells in $\CP_{n+1}$. Observe that
\begin{itemize}
    \item The number of neighbours of a particular vertex $\bar{v}_0$ in $\bar{\sigma}$ is same as number of neighbours of $v_0$ in $\sigma$.
    \item The neighbours of vertex $v_0$ in $\sigma$ are naturally in 1-1 correspondence with the neighbours $r(v_0)$ in $r(\sigma)$.
\end{itemize}

\begin{theorem}
Let $\sigma= (I_1,I_2, \dots, I_k,I_{\omega})$ and $\tau= r(\sigma)=(I_k,I_{k-1}, \dots, I_1,I_{\omega})$. 
Then the difference in the orientations induced on the cell $\bar{\sigma}=\bar{\tau}$ in $\QP_{n+1}$ by $\sigma$ and $\tau$ in $\CP_{n+1}$ is given by the expression
\begin{equation}\label{sign correction}
    \mathrm{sgn}(|A|-|I_{\omega}|+1) \cdot \mathrm{sgn}(|I_{\omega}|-1) \cdot \displaystyle\prod_{i=1}^{k} \mathrm{sgn}(|I_i|)  \cdot \mathrm{sgn}(|I_{\omega}|),
\end{equation}
where $|A|$ is the number of neighbors of $\mathrm{PV}(\sigma)$.
\end{theorem}
\vspace{-1cm}
\begin{proof}
    Without loss of generality assume $PV(\sigma)=(1,2,\dots,n+1)$ i.e., the blocks $I_1=\{1,2,\dots,a_1\},\dots, I_j=\{a_{j-1}+1,a_{j-1}+2,\dots,a_j\}$ for every $j$ such that $\ 1\leq j \leq k$ or $j=\omega$.
    
    The neighbors of $\mathrm{PV}(\sigma)$ are ordered as follows.
	    \begin{equation}\label{list1}
	    \begin{gathered}
	    v_0=(\{1\}\{2\}\{3\}\dots\{n+1\})\\
	    v_1=(\{2\}\{1\}\{3\}\dots\{n+1\})\\
	    v_2=(\{1\}\{3\}\{2\}\dots\{n+1\})\\
	    \vdots\\
	    v_{a_1-1}=(\{1\}\dots \{a_1\}\{a_1-1\}\dots\{n+1\})\\
	    \vdots\\
	    v_{a_k-k-1}=(\{1\} \dots \{a_k-1\}\{a_k\}\dots \{n+1\})\\
	    v_{a_k-k}=(\{1\} \dots \{a_k+2\}\{a_k+1\}\dots \{n+1\})\\
	    \vdots\\
	    v_{|A|}=(\{1\}\{2\} \dots\{n+1\}\{n\})\\
	    \end{gathered}
	    \end{equation}
	    
    The neighbors of $\mathrm{PV}(\tau)$ are ordered as follows.
	\begin{equation}\label{list2}
	    \begin{gathered}
	    w_0=(\{a_{k-1}+1\}\{a_{k-1}+2\} \dots \{a_k\}\{a_{k-2}\} \dots \{n+1\})\\
	    w_1=(\{a_{k-1}+2\}\{a_{k-1}+1\} \dots \{a_k\}\{a_{k-2}\} \dots \{n+1\})\\
	    \vdots\\
	    w_{a_k-a_{k-1}}=(\{a_{k-1}+1\}\{a_{k-1}+2\} \dots \{a_{k-2}+1\}\{a_{k-2}\}\dots \{n+1\})\\
	    \vdots\\
	    w_{|A|-a_k-a_{k-1}-1}=(\{a_{k-1}+1\}\{a_{k-1}+2\} \dots \{2\}\{1\} \dots\{n+1\})\\
	    \vdots\\
	    w_{|A|-a_k-1}=(\{a_{k-1}+1\}\{a_{k-1}+2\} \dots \{a_1\}\{a_1-1\}\dots\{n+1\})\\
	    w_{|A|-a_k}=(\{a_{k-1}+1\}\{a_{k-1}+2\} \dots \{a_1\}\{a_k+2\}\{a_k+1\}\dots\{n+1\})\\
	    \vdots\\
	    w_{|A|}=(\{a_{k-1}+1\}\{a_{k-1}+2\} \dots \{n+1\}\{n\})
	    \end{gathered}
	\end{equation}		
    
    Now apply $r$ on the neighbours of $\mathrm{PV}(\tau)$ to obtain an ordered collection of vertices in $\sigma$. This would enable us to compare the orientations induced by \ref{list1} and \ref{list2} on $\bar{\sigma}$.
	\begin{equation}
	    \begin{gathered}
	    r(w_0)=(\{n\}\{n-1\} \dots \{a_k+1\}\{a_1\} \dots\{a_{k-1}+1\} \{n+1\})\\
	    r(w_1)=(\{n\}\{n-1\} \dots \{a_k+1\}\{a_1\} \dots \{a_{k-1}+1\}\{a_{k-1}+2\})\\
	    \vdots\\
	    r(w_{a_k-a_{k-1}})=(\{n\}\{n-1\} \dots \{a_{k-2}\}\{a_{k-2}+1\}\dots \{a_{k-1}+1\}\{n+1\})\\
	    \vdots\\
	    r(w_{|A|-a_k-a_{k-1}-1})=(\{n\}\{n-1\} \dots \{1\}\{2\} \dots\{a_{k-1}+1\}\{n+1\})\\
	    \vdots\\
	    r(w_{|A|-a_k-1})=(\{n\}\{n-1\} \dots \{a_1-1\}\{a_1\}\dots\{a_{k-1}+1\}\{n+1\})\\
	    r(w_{|A|-a_k})=(\{n\}\{n-1\} \dots \{a_1+1\}\{a_k+2\}\{a_k\}\dots\{a_{k-1}+1\}\{n+1\})\\
	    \vdots\\
	    r(w_{|A|})=(\{n-1\}\{n-2\} \dots \{a_{k-1}+1\}\{n\}\{n+1\})\\
	    \end{gathered}
	\end{equation}

	Let $\xi$ be the permutation which takes the vertex $\mathrm{PV}(\sigma)$ to the vertex $r(\mathrm{PV}(\tau))$.
	
	\begin{equation}
	    \begin{gathered}
	    \xi(v_0)=(\{a_1\}\{a_1-1\} \dots \{2\}\{1\}\{a_2\} \dots\{n+1\}\dots \{a_k+1\})\\
	    \xi(v_1)=(\{a_1\}\{a_1-1\} \dots \{1\}\{2\} \dots \{n+1\}\dots\{a_k+1\})\\
	    \xi(v_2)=(\{a_1\}\{a_1-1\} \dots \{2\}\{3\}\{1\} \dots \{n+1\}\dots\{a_k+1\})\\
	    \vdots\\
	    \xi(v_{a_1-1})=(\{a_1-1\}\{a_1\} \dots \{2\}\{1\}\dots \{n+1\}\dots\{a_k+1\})\\
	    \vdots\\
	    \xi(v_{a_k-k-1})=(\{a_1\}\{a_1-1\} \dots \{a_k\}\{a_k-1\} \dots\{n+1\}\{a_k+1\})\\
	    \xi(v_{a_k-k})=(\{a_1\}\{a_1-1\} \dots \{a_k+1\}\{a_k+2\} \dots\{n+1\}\{a_k+1\})\\
	    \vdots\\
	    \xi(v_{|A|})=(\{a_1\}\{a_1-1\} \dots \{n\}\{n+1\}\dots \{a_k+1\})\\
	    \end{gathered}
	\end{equation}	

It is clear from above that if $\sigma= (I_1,I_2, \dots, I_k,I_{\omega})$ the $\mathrm{sign}$ of the permutation $\xi$ is $\displaystyle\prod_{i=1}^{k} \mathrm{sgn}(|I_i|) \cdot \mathrm{sgn}(|I_{\omega}|)$. The $\mathrm{sign}$ of permutation coming from comparing the induced orientations on $\sigma \in \CP_{n+1}$ by the vertices $v_0$ and $\xi(w_0)$ is $\mathrm{sgn}(|A|-|I_{\omega}|+1) \cdot \mathrm{sgn}(|I_{\omega}|-1)$. 
Thus the total $\mathrm{sign}$ to be taken into account is $\mathrm{sgn}(|A|-|I_{\omega}|+1) \cdot \mathrm{sgn}(|I_{\omega}|-1) \cdot  \displaystyle\prod_{i=1}^{k} \mathrm{sgn}(|I_i|) \cdot \mathrm{sgn}(|I_{\omega}|).$
\end{proof}


The following observations are helpful in computing the $\mathbb{Z}$-homology of $\QP_{n+1}$.

\begin{enumerate}
    \item There exists no path or exactly two paths between critical cells whose dimension differ by one. Thus the matrices corresponding to the boundary maps contain only 2's and 0's depending on whether the orientation induced by the paths match or not.
    \item These are good paths, except some paths involves a identification of a cell $\sigma$ with $r(\sigma)$, where the orientation change involved is given by \cref{sign correction}.
\end{enumerate}

\begin{definition}
    Two rectangular matrices $A,B \in M_{n\times m}(\mathbb{Z})$ are called equivalent if they can be transformed into one another by a combination of elementary row and column operations.
\end{definition}

\begin{definition}[2-full rank]
    Let $f: \mathbb{Z}^{m} \rightarrow \mathbb{Z}^{n}$ be a group homomorphism. The map $f$ is \emph{2-full rank}, denoted $2\cF$, if it is equivalent to a matrix with only $2$'s on the main diagonal and $0$ everwhere elese.
\end{definition}

\begin{proposition}
The boundary maps in the Morse complex of $\QP_{n+1}$ are either 2-full rank or null maps. \textit{i.e.,} if the Morse complex $\mathcal{M}_{\boldsymbol{\cdot}}$ on $\QP_{n+1}$ is
$$0\longrightarrow \mathcal{M}_{n-2} \mathop{\longrightarrow}^{\tilde{\partial}_{n-2}} \mathcal{M}_{n-3} \mathop{\longrightarrow}^{\tilde{\partial}_{n-3}}  \dots  \mathop{\longrightarrow}^{\tilde{\partial}_1} \mathcal{M}_0 \rightarrow 0$$
then the boundary maps \[ \tilde{\partial}_i \equiv
	\begin{cases}
	\mathbf{0}, & \mathrm{if\ } i\mathrm{\ is\ odd};\\
	2\cF, & \mathrm{if\ } i\mathrm{\ is\ even}.
	\end{cases}
\]
\end{proposition}
\begin{proof}
If the sign correction for the identification involved in the path is positive (resp. negative), then by \cref{good path lemma} and \cref{paths are good} the coefficient $\l \partial \alpha, \beta \r =0$ (resp. $\l \partial \alpha, \beta \r =2$). 

\textbf{Claim 1:} $\tilde{\partial}_1=0$.

\textit{Proof.}
    Let $\alpha = (i,I,\nabla,N)$ and $\beta=(j,J,\nabla', N')$ be two critical cells contained in $\cM_1$ and $\cM_0$ respectively.
    
    \begin{enumerate}
        \item If $|N|=1$, then $|I|=2$ and $N'=N=\{n+1\}$. Otherwise there will be no path between the cells giving $\l \tilde{\partial} \alpha,\beta \r =0$.
        There is an identification of the cell $(i,I-j,j,\nabla,N)$ with its image under the map $r$ during the path. All the blocks of this cell are singletons, thus the sign correction given by the \cref{sign correction} is 1.
    
        \item If $|N|=2$, then $|I|=1$, $N' \subset N$ and $|N'|=1$. Otherwise an argument similar to above shows that $\l \tilde{\partial} \alpha,\beta \r =0$.
        There is an identification of the cell  $(i,I,\nabla,N-t,t)$ or $(i,I,\nabla,t,N-t)$  with its image under the map $r$ during the path. All the blocks of these cells are singletons, thus the sign correction given by the \cref{sign correction} is 1.
    \end{enumerate}
    Since there is no effect on the orientation induced along the paths by the action, an argument similar to \cref{boundary vanish} shows that $\tilde{\partial}_1 =0$.

\textbf{Claim 2:} $\partial_i=0$ when $i$ is odd.
    
\textit{Proof.} Let $\alpha = (i,I,\nabla,N)$ and $\beta=(j,J,\nabla', N')$ be two critical cells contained in $\cM_i$ and $\cM_{i-1}$ respectively. Also, let $d$ denote the number of blocks in $\alpha$ which is equal to $(n+1-i)$. 
\begin{enumerate}
    \item Let $N=N'$ and $J=I-j$.
    \begin{itemize}
        \item Let $|N|=k+1$ for some $k$ odd. There is an identification of the cell $(i,I-j,j,\nabla,N)$ with its image under $r$ during the path. From \cref{sign correction}, the sign correction $\cS$ is given by $\mathrm{sgn}(|J|)\cdot \mathrm{sgn}(|N|)\cdot \mathrm{sgn}(|J|-1) \cdot \mathrm{sgn} (|N|-1)$.\\
        Observe that the $|J|=n-k-d-3=(n+1-d)-(k+4)$. Since $(n+1-d)=i$ is odd and $(k+4)$ is odd, $|J|$ is even and 
        \begin{align*}
            \mathrm{sgn}(|J|) &=(-1)^{\frac{n-k-d-3}{2}}.
        \end{align*}
        Similarly,
        \begin{align*}
            \mathrm{sgn}(|N|) &= (-1)^{\frac{k+1}{2}},\\
            \mathrm{sgn}(|J|-1) &= (-1)^{\frac{n-k-d-5}{2}},\\
            \mathrm{sgn} (|N|-1) &=(-1)^{\frac{k-1}{2}}.
        \end{align*}
        This shows that the sign correction $\cS$ is equal to 1.
            
        \item Let $|N|=k+1$ for some $k$ even. There is an identification of the cell $(i,I-j,j,\nabla,N)$ with its image under $r$ during the path. From \cref{sign correction}, the sign correction $\cS$ is given by $\mathrm{sgn}(|J|)\cdot \mathrm{sgn}(|N|)\cdot \mathrm{sgn}(|J|-1) \cdot \mathrm{sgn} (|N|-1)$.\\
        Observe that the $|J|=n-k-d-3=(n+1-d)-(k+4)$. Since $(n+1-d)=i$ is odd and $(k+4)$ is even, $|J|$ is odd and 
        \begin{align*}
            \mathrm{sgn}(|J|) &=(-1)^{\frac{n-k-d-4}{2}}.
        \end{align*}
        Similarly,
        \begin{align*}
            \mathrm{sgn}(|N|) &= (-1)^{\frac{k}{2}},\\
            \mathrm{sgn}(|J|-1) &= (-1)^{\frac{n-k-d-4}{2}},\\
            \mathrm{sgn} (|N|-1) &=(-1)^{\frac{k}{2}}.
        \end{align*}
        This shows that the sign correction $\cS$ is equal to 1.
        \end{itemize}
        
    \item Let $I=J$ and $N'=N-t$ for some $t \in N$.
    \begin{itemize}
        \item Let $|N|=k+1$ for some $k$ odd. There is an identification of the cell  $(i,I,\nabla,N-t,t)$ or $(i,I,\nabla,t,N-t)$  with its image under the map $r$ during the path. From \cref{sign correction}, the sign correction $\cS$ is given by $\mathrm{sgn}(|I|)\cdot \mathrm{sgn}(|N'|)\cdot \mathrm{sgn}(|I|-1) \cdot \mathrm{sgn} (|N'|-1)$.\\
        Observe that the $|I|=n-k-d-2=(n+1-d)-(k+3)$. Since $(n+1-d)$ is odd and $(k+3)$ is even, $|I|$ is odd and 
        \begin{align*}
            \mathrm{sgn}(|I|) &=(-1)^{\frac{n-k-d-3}{2}}.
        \end{align*}
        Similarly,
        \begin{align*}
            \mathrm{sgn}(|N|) &= (-1)^{\frac{k-1}{2}},\\
            \mathrm{sgn}(|J|-1) &= (-1)^{\frac{n-k-d-3}{2}},\\
            \mathrm{sgn} (|N|-1) &=(-1)^{\frac{k-1}{2}}.
        \end{align*}
        Clearly, the sign correction $\cS$ is equal to 1.
            
        \item Let $|N|=k+1$ for some $k$ even. There is an identification of the cell  $(i,I,\nabla,N-t,t)$ or $(i,I,\nabla,t,N-t)$  with its image under the map $r$ during the path. From \cref{sign correction}, the sign correction $\cS$ is given by $\mathrm{sgn}(|I|)\cdot \mathrm{sgn}(|N'|)\cdot \mathrm{sgn}(|I|-1) \cdot \mathrm{sgn} (|N'|-1)$.\\
        Observe that the $|I|=n-k-d-2=(n+1-d)-(k+3)$. Since $(n+1-d)$ is odd and $(k+3)$ is odd, $|I|$ is even and 
        \begin{align*}
            \mathrm{sgn}(|I|) &=(-1)^{\frac{n-k-d-2}{2}}.
        \end{align*}
        Similarly,
        \begin{align*}
            \mathrm{sgn}(|N|) &= (-1)^{\frac{k}{2}},\\
            \mathrm{sgn}(|J|-1) &= (-1)^{\frac{n-k-d-4}{2}},\\
            \mathrm{sgn} (|N|-1) &=(-1)^{\frac{k-2}{2}}.
        \end{align*}
        Clearly, the sign correction $\cS$ is equal to 1. 
    \end{itemize}
\end{enumerate}

\textbf{Claim 3:} $\partial_i=2\cF$ when $i$ is even.

\textit{Proof.} Let $\alpha = (i,I,\nabla,N)$ and $\beta=(j,J,\nabla', N')$ be two critical cells contained in $\cM_i$ and $\cM_{i-1}$ respectively.  Also, let $d$ denote the number of blocks in $\alpha$ which is equal to $(n+1-i)$. It is enough to show that $\l \partial \alpha, \beta \r =2$ whenever $N=N'$.
    \begin{enumerate}
        \item Let $|N|=k+1$ for some $k$ odd. There is an identification of the cell $(i,I-j,j,\nabla,N)$ with its image under $r$ during the path. From \cref{sign correction}, the sign correction $\cS$ is given by $\mathrm{sgn}(|J|)\cdot \mathrm{sgn}(|N|)\cdot \mathrm{sgn}(|J|-1) \cdot \mathrm{sgn} (|N|-1)$.\\
            Observe that the $|J|=n-k-d-3=(n+1-d)-(k+4)$. Since $(n+1-d)$ is even and $(k+4)$ is odd, $|J|$ is odd and 
            \begin{align*}
                \mathrm{sgn}(|J|) &=(-1)^{\frac{n-k-d-4}{2}}.
            \end{align*}
            Similarly,
            \begin{align*}
                \mathrm{sgn}(|N|) &= (-1)^{\frac{k+1}{2}},\\
                \mathrm{sgn}(|J|-1) &= (-1)^{\frac{n-k-d-4}{2}},\\
                \mathrm{sgn} (|N|-1) &=(-1)^{\frac{k-1}{2}}.
            \end{align*}
            Clearly, the sign correction $\cS$ is equal to -1.
            
        \item Let $|N|=k+1$ for some $k$ even. There is an identification of the cell $(i,I-j,j,\nabla,N)$ with its image under $r$ during the path. From \cref{sign correction}, the sign correction $\cS$ is given by $\mathrm{sgn}(|J|)\cdot \mathrm{sgn}(|N|)\cdot \mathrm{sgn}(|J|-1) \cdot \mathrm{sgn} (|N|-1)$.
            Observe that $|J|=n-k-d-3=(n+1-d)-(k+4)$. Since $(n+1-d)$ is even and $(k+4)$ is even, $|J|$ is even and 
            \begin{align*}
                \mathrm{sgn}(|J|) &=(-1)^{\frac{n-k-d-3}{2}}.
            \end{align*}
            Similarly,
            \begin{align*}
                \mathrm{sgn}(|N|) &= (-1)^{\frac{k}{2}},\\
                \mathrm{sgn}(|J|-1) &= (-1)^{\frac{n-k-d-5}{2}},\\
                \mathrm{sgn} (|N|-1) &=(-1)^{\frac{k}{2}}.
            \end{align*}
    \end{enumerate}
        Clearly, the sign correction $\cS$ is -1.
\end{proof}

\begin{theorem}\label{mainthm}
The $\mathbb{Z}$-homology of $\QP_{n+1}$ is given as follows.

If n is even, then 
\[ H_i(\QP_{n+1},\mathbb{Z}) =
	\begin{cases}
    \displaystyle \bigoplus_{\xi(n,i)}\mathbb{Z}_2 , & \mathrm{if\ } i \mathrm{\ is\ odd\ and\ } 0 \leq i \leq n-2 ;\\
	\displaystyle \bigoplus_{\binom{n}{i}}\mathbb{Z} , & \mathrm{if\ } i \mathrm{\ is\ even\ and\ } 0 \leq i \leq n-2;\\
	0, & \mathrm{otherwise.}
	\end{cases}
\]

If n is odd, then

\[ H_i(\QP_{n+1},\mathbb{Z}) =
	\begin{cases}
	\displaystyle \bigoplus_{\xi(n,i)}\mathbb{Z}, & \mathrm{if\ } i = n-2;\\
	\displaystyle \bigoplus_{\xi(n,i)}\mathbb{Z}_2, & \mathrm{if\ } i \mathrm{\ is\ odd\ and\ } 0 \leq i < n-2 ;\\
	\displaystyle \bigoplus_{\binom{n}{i}}\mathbb{Z} , & \mathrm{if\ } i \mathrm{\ is\ even\ and\ } 0 \leq i \leq n-2;\\
	0, & \mathrm{otherwise.}
	\end{cases}
\]

Where, $\xi(n,i)$ denotes the sum $\displaystyle \sum_{k=0}^{i} \binom{n}{k}$.
\end{theorem}

\begin{proof}
We will present the proof for the case of $n$ being odd, the proof for the even case is similar in nature.
\begin{enumerate}
    \item If $i = n-2$, then the Morse complex looks like 
    $$ 0 \mathop{\longrightarrow}^{0} \mathcal{M}_{n-2}\mathop{\longrightarrow}^{0} \mathcal{M}_{n-3}.$$
    Then the homology at $\mathcal{M}_{n-2}$ is $\mathbb{Z}^{\mathrm{rnk}(\mathcal{M}_{n-2})}$.
    
    \item If $i$ is odd and $i \neq n-2$, then the Morse complex looks like 
    $$\mathcal{M}_{i+1} \mathop{\longrightarrow}^{2F} \mathcal{M}_{i}\mathop{\longrightarrow}^{0} \mathcal{M}_{i-1}.$$
    Then the homology at $\mathcal{M}_{i}$ is $\mathbb{Z}_2^{ \mathrm{rnk}(\mathcal{M}_{i})}$.
    
    \item If $i$ is even, then the Morse complex looks like 
    $$\mathcal{M}_{i+1} \mathop{\longrightarrow}^{0} \mathcal{M}_{i}\mathop{\longrightarrow}^{2F} \mathcal{M}_{i-1}.$$
    Then the homology at $\mathcal{M}_{i}$ is $\mathbb{Z}^{\mathrm{rnk}(\mathcal{M}_{i})-\mathrm{rnk}(\mathcal{M}_{i-1})}$.
\end{enumerate}
\end{proof}

\printbibliography
\end{document}